\documentclass[11pt]{amsart}
\usepackage{amsmath,graphics}
\usepackage{amsfonts,amssymb, enumerate, color}
\usepackage{ulem}

\usepackage{geometry}
\usepackage[latin1]{inputenc}

\usepackage[T1]{fontenc}
\usepackage[all,cmtip]{xy}

\usepackage{lipsum}
\newcounter{quotecount}

\usepackage{eucal} 
\usepackage{cmmib57} 

\theoremstyle{plain}
\newtheorem{theorem}{Theorem}[section]
\newtheorem{lemma}[theorem]{Lemma}
\newtheorem{corollary}[theorem]{Corollary}
\newtheorem{proposition}[theorem]{Proposition}

\newtheorem{remark}[theorem]{Remark}
\newtheorem{definition}[theorem]{Definition}

\newtheorem{example}[theorem]{Example}

\theoremstyle{definition}

\newcommand{\MC}{{\mathcal M}}

\newcommand{\RC}{{\mathcal R}}
\def\BB{\mathcal{B}}

\newcommand{\sg}{\mathfrak{g}}

\newcommand{\Z}{\mathbb{Z}}

\newcommand{\R}{\mathbb{R}}
\newcommand{\C}{\mathbb{C}}

\newcommand{\bu}{\mathbf{u}}

\renewcommand{\Re}{\operatorname{Re}}
\renewcommand{\Im}{\operatorname{Im}}

\def\a{\alpha}
\def\G{\Gamma}
\def\ng{{\mathcal{N}(\G)}}

\def\cp{\mathbb{CP}}

\def\g{\gamma}
\def\ga{\Gamma}
\def\part{\partial}

\def\Hom{\mbox{Hom}}

\def\im{\mbox{Image}}
\def\rk{\mbox{rank}}

\def\ker{\mbox{Ker}}

\def\tbigpm{[t_1^{\pm 1},\dots,t_r^{\pm 1}]}

\def\wti{\widetilde}

\def\lra{\longrightarrow}

\def\bd{\begin{definition}}
\def\ed{\end{definition}}
\def\bt{\begin{theorem}}
\def\et{\end{theorem}}
\def\br{\begin{remark}}
\def\er{\end{remark}}
\def\bc{\begin{corollary}}
\def\ec{\end{corollary}}
\def\bp{\begin{proposition}}
\def\ep{\end{proposition}}
\def\be{\begin{equation}}
\def\ee{\end{equation}}
\def\bn{\begin{enumerate}}
\def\en{\end{enumerate}}
\def\ba{\begin{array}}
\def\ea{\end{array}}
\def\bex{\begin{example}}
\def\eex{\end{example}}

\begin{document}

\title[]{Topology of %smooth closed 
subvarieties of complex semi-abelian varieties
% and vanishing results %and the signed Euler characteristic property
}

\author{Yongqiang Liu}
\address{Department of Mathematics, KU Leuven, 
Celestijnenlaan 200B, B-3001 Leuven, Belgium} 
\email{liuyq1117@gmail.com}
\author{Laurentiu Maxim}
\address{Department of Mathematics,
          University of Wisconsin-Madison,
          480 Lincoln Drive, Madison WI 53706-1388, USA.}
\email {maxim@math.wisc.edu}
\author{Botong Wang}
\address{Department of Mathematics,
          University of Wisconsin-Madison,
          480 Lincoln Drive, Madison WI 53706-1388, USA.}
\email {bwang274@wisc.edu}
\thanks{}

\date{\today}



\begin{abstract}
We use the non-proper Morse theory of Palais-Smale to investigate the topology of smooth closed subvarieties of complex semi-abelian varieties, and that of their infinite cyclic covers. As main applications, we obtain the finite generation (except in the middle degree) of the corresponding integral Alexander modules, as well as the signed Euler characteristic property and generic vanishing for rank-one local systems on such subvarieties. Furthermore, we give a more conceptual (topological) interpretation of the signed Euler characteristic property in terms of vanishing of Novikov homology.  As a byproduct, we prove a generic vanishing result for the $L^2$-Betti numbers of very affine manifolds.  Our methods also recast June Huh's extension of Varchenko's conjecture to very affine manifolds, and provide a generalization of this result in the context of smooth closed sub-varieties of semi-abelian varieties.
\end{abstract}

\maketitle

\tableofcontents

%=====================================
\section{Introduction}

A complex {\it semi-abelian variety} is a complex algebraic group $G$ which is an extension
$$1 \to T \to G \to A \to 1,$$
where $A$ is an abelian variety and $T\cong(\C^*)^m$ is an algebraic torus.
In this paper we study the topology of smooth closed subvarieties of semi-abelian varieties.

It was shown in \cite{FK00} that for a smooth closed  $n$-dimensional subvariety $X$ of a semi-abelian variety, the topological Euler characteristic of $X$ is signed, i.e.,
\be\label{e1} (-1)^n \cdot \chi(X) \geq 0.\ee
On the other hand, (\ref{e1}) fails to be true if $X$ is singular, e.g., see \cite{BWb}. The proof of (\ref{e1}) in loc.cit. makes use of characteristic cycles and the language of perverse sheaves. In fact, (\ref{e1}) is a direct consequence of a more general fact  that the Euler characteristic $\chi(G;\mathcal{P})$ of a perverse sheaf on a semi-abelian variety is non-negative.
In the case of abelian varieties, the latter also follows from the generic vanishing theorems for perverse sheaves, see e.g., \cite{BSS,Sc15}, while the corresponding statement for the Euler characteristic of perverse sheaves on a complex algebraic torus was proved in \cite{LS} (see also \cite{GL}).

One of the goals of this paper is to give a completely topological proof and interpretation of the signed Euler characteristic property  (\ref{e1}). Our arguments are based on the non-proper Morse theory developed by Palais-Smale in \cite{PS64}, and should serve as an access point for geometers and topologists who are not necessarily accustomed with the language of perverse sheaves.

\medskip

One of the main results of this paper is the following (see Corollary \ref{topologyone}):
\bt\label{ti1}
Let $X$ be an $n$-dimensional closed connected smooth subvariety of a complex semi-abelian variety. For a generic group homomorphism $\xi: \pi_1(X)\to \Z$, the corresponding infinite cyclic cover $X^\xi$ is homotopy equivalent to a finite CW-complex with possibly infinitely many $n$-cells attached.  Moreover, if $\xi$ is surjective, then $X^\xi$ is connected.
\et

An irreducible closed subvariety of a complex algebraic torus is called a {\it very affine variety}. These are heavily studied in tropical geometry and algebraic statistics, e.g., see \cite{HS}. Examples of very affine varieties include complements in the complex affine space of essential hyperplane arrangements, as well as complements of complex toric hyperplane arrangements. The signed Euler characteristic property for smooth very affine varieties was proved in \cite{Hu} as a consequence of the generalized Varchenko conjecture; see also \cite{BWb} for a different approach.

In the case of smooth very affine varieties, Theorem \ref{ti1} can be refined as follows (see Corollary \ref{fibrb} and Corollary \ref{fibrbb}):
\bt\label{ti2} 
Let $X$ be an $n$-dimensional smooth connected very affine variety in the complex affine torus $G$.  Then for a generic homomorphism $\xi: \pi_1(X)\to \Z$, which  factors through $\pi_1(X) \to \pi_1(G)$, there exists a continuous map $F_\xi: X_0\to S^1$ defined on a subset $X_0$ of $X$, such that the following properties hold: 
\begin{enumerate}
\item $F_\xi$ is a fibration whose fiber is of finite homotopy type,
\item $X$ is homotopy equivalent to $X_0$ with $(-1)^n \chi(X)$ $n$-cells attached. 
\end{enumerate}
Moreover, if the mixed Hodge structure of $H^1(X; \Z)$ is pure of type $(1,1)$, or equivalently, there exists a smooth compactification $\bar{X}$ of $X$ such that $b_1(\bar{X})=0$, the above statement holds for  any generic homomorphism $\xi: \pi_1(X)\to \Z$. 
\et
The purity assumption from the statement of Theorem \ref{ti2} is satisfied for the main examples of very affine varieties mentioned above: complements of essential hyperplane arrangements and resp. complements of toric hyperplane arrangements.

\medskip

As one of the main topological applications of Theorem \ref{ti1} we mention the following:
\bt\label{ti3}  (finite generation of integral Alexander modules)\\
Let $X$ be an $n$-dimensional closed connected smooth subvariety of a complex semi-abelian variety. For a generic group homomorphism $\xi: \pi_1(X)\to \Z$, the corresponding integral Alexander modules $H_i(X^\xi;\Z)$ are finitely generated abelian groups for any $i \neq n$.
\et

Theorem \ref{ti3} implies immediately the following:
\bc (torsion property) \\
Let $X$ be an $n$-dimensional closed connected smooth subvariety of a complex semi-abelian variety. For a generic group homomorphism $\xi: \pi_1(X)\to \Z$, the corresponding complex Alexander modules $H_i(X^{\xi};\C)$ are torsion $\C[t,t^{-1}]$-modules, for all $i \neq n$.\ec

Under the assumptions of the previous corollary, let $\Delta^{\xi}_i(t)$ denote the  Alexander polynomial associated to $H_i(X^{\xi};\C)$. We also let $\Delta^{\xi}_n(t)$ be the Alexander polynomial associated to the torsion part of the finitely generated $\C[t,t^{-1}]$-module $H_n(X^{\xi};\C)$.  Then by using \cite[Proposition 1.4]{BLW} we have the following:

\bc (roots of Alexander polynomials)\\
For any $i \geq 0$, the zeros of the Alexander polynomials $\Delta^{\xi}_i(t)$ are roots of unity.
\ec

Furthermore, Theorem \ref{ti3} also has the following important consequences:
\bc \label{ci1} (signed Euler characteristic)\\
If $X$ is an $n$-dimensional closed connected smooth subvariety of a complex semi-abelian variety, then $ (-1)^n \cdot \chi(X) \geq 0$.
\ec

\bc\label{ci2} (generic vanishing)\\
Let $X$ be an $n$-dimensional closed connected smooth subvariety of a complex semi-abelian variety, and fix a generic epimorphism $\xi: \pi_1(X)\to \Z$. Given any $s\in \C^*$, denote by $L^{\xi}_s$ the rank-one $\C$-local system on $X$ whose monodromy representation is given by 
$$\pi_1(X)\xrightarrow{\xi} \Z\to \C^*,$$
where the homomorphism $ \Z\to \C^*$ is defined by $1\mapsto s$. Then for all but finitely many $s\in \C^*$, we have:
$$H_i(X, L^{\xi}_s)=0, \text{ for all } i\neq n.$$
\ec

A more conceptual (and purely topological) interpretation of the signed Euler characteristic property (\ref{e1}) can be given in terms of Novikov homology. 
More precisely, to any pair $(X,\xi)$ as above, one can associate Novikov-Betti numbers $b_i(X,\xi)$ and resp. Novikov-torsion numbers $q_i(X,\xi)$. Then the following result holds:

\bc\label{ci3} (vanishing of Novikov homology)\\
Let $X$ be an $n$-dimensional closed connected smooth subvariety of a complex semi-abelian variety, and fix a generic epimorphism $\xi: \pi_1(X)\to \Z$ as above. Then
$b_i(X,\xi)=q_i(X,\xi)=0$ for any $i \neq n$, and $b_n(X,\xi)=(-1)^n\chi(X)$.
\ec

%For another topological interpretation of (\ref{e1}) in terms of $L^2$-Betti numbers, see Remark \ref{l2r}.  I suggest that we add the claim about $L^2$-betti number for very affine manifold here, since this is the only corollary which could not be obtained by the perverse sheaves technique.

As an application of Theorem \ref{ti2}, we also obtain the following generic vanishing for the $L^2$-Betti numbers of very affine manifolds. To the authors' knowledge, this result seems to be new, and it provides a more unifying treatment of the corresponding assertions for toric arrangement complements, see \cite{DS} (which seems to contain a  gap in the proof, cf. \cite{D15}) and, respectively, essential hyperplane arrangement complements, see \cite{DJL}.

If $\G$ is a countable group, and the pair $(X,\xi)$ is as above, an epimorphism $\a:\pi_1(X) \to \G$ is called $\xi$-admissible if $\xi$ factors through $\a$. Then we have:

\bc \label{L2} (Vanishing of $L^2$-Betti numbers of very affine manifolds) \\ 
Let $X$ be an $n$-dimensional smooth connected very affine variety in the complex affine torus $G$. Fix a generic epimorphism $\xi: \pi_1(X)\to \Z$, which factors through $\pi_1(X) \to \pi_1(G)$.  Then, for any $\xi$-admissible epimorphism $\a:\pi_1(X) \to \G$, we have that  $b_i^{(2)}(X,\a) =0$ for all $i \neq n$, and $b_n^{(2)}(X,\a) =(-1)^n\chi(X)$.
\ec

Our methods also recast June Huh's extension of Varchenko's conjecture to very affine manifolds, see \cite[Theorem 1.1]{Hu}, and provide a generalization of this result in the context of smooth closed sub-varieties of semi-abelian varieties.

\bt\label{Hu}\cite{Hu}
Let $X$ be  an $n$-dimensional closed connected smooth subvariety of  $(\C^*)^m$, and let $(z_1, \ldots, z_m)$ be the coordinates of $(\C^*)^m$. Given $\mathbf{u}=(u_1, \ldots, u_m)\in \Z^m$, let $\psi_{\bu}$ be the restriction of the master function $z_1^{u_1}\cdots z_m^{u_m}$ to $X$. If $\bu$ is sufficiently general, then the following properties hold: 
\begin{enumerate}
\item The function $\psi_\bu$ has only finitely many critical points on $X$.
\item All critical points of $\psi_\bu$ are regular. 
\item The number of critical points is equal to the signed Euler characteristic 
$$(-1)^n\chi(X).$$
\end{enumerate}
\et

Notice that the function $\psi_\bu$ has a critical point at $x\in X$ if and only if the logarithmic 1-form 
$$\frac{d\psi_\bu}{\psi_\bu}=u_1\frac{dz_1}{z_1}+\cdots+u_m\frac{dz_m}{z_m}$$
degenerates at $x$. Moreover, $\psi_\bu$ has a regular critical point at $x$ if and only if $\frac{d\psi_\bu}{\psi_\bu}$ has a regular singularity at $x$. 

The logarithmic forms $\frac{d\psi_\bu}{\psi_\bu}$ are well-defined for $\bu\in \C^m$, and they are left invariant holomorphic $1$-forms on $(\C^*)^m$. Conversely, every left invariant holomorphic 1-form on $(\C^*)^m$ is equal to $\frac{d\psi_\bu}{\psi_\bu}$ for some $\bu\in \C^m$. Thus, Theorem \ref{Hu} can be formulated by using left invariant holomorphic $1$-forms on the affine torus. In this setting, we obtain the following generalization of Huh's theorem to smooth subvarieties of semi-abelian varieties. 
\bt\label{huhgen}
Let $G$ be a complex semi-abelian variety, and let $X$ be an $n$-dimensional closed connected smooth subvariety of $G$. Let $\Gamma$ be the space of left invariant holomorphic $1$-forms on $G$. Given $\eta\in \Gamma$, we denote the restriction of $\eta$ to $X$ by $\eta_X$. If $\eta\in \Gamma$ is general, then the following statements hold: 
\begin{enumerate}
\item The holomorphic 1-form $\eta_X$ degenerates at finitely many points in $X$.
\item At all degeneration points, $\eta_X$ has regular singularity. 
\item The number of degeneration points of $\eta_X$ is equal to the signed Euler characteristic 
$$(-1)^n \chi(X).$$
\end{enumerate}
\et

Note that the signed Euler characteristic $(-1)^n \chi(X)$ of $X$ as in the above theorem can also be identified with the Gaussian degree of $X$, see \cite{FK00}[Corollary 1.5].

\medskip

Let us briefly describe the structure of the paper. In Section \ref{pre} we develop the topological tools needed for proving Theorem \ref{ti3} and its consequences (Corollaries \ref{ci1}-\ref{L2}). In Section \ref{np} we recall the main aspects of the non-proper Morse theory of Palais-Smale, and provide a circle-valued version of this theory. In Section \ref{mr}, we introduce the main constructions of the paper, and prove Theorem \ref{ti1} (see Corollary \ref{topologyone}), Theorem \ref{ti2} (see Corollary \ref{fibrb} and Corollary \ref{fibrbb}) and Theorem \ref{huhgen}. Sections \ref{mp} and \ref{PS} are of technical nature, and supply the proofs of the Morse-theoretic statements behind Theorems \ref{ti1},  \ref{ti2} and \ref{huhgen}. 

%%%%%%%%%%%%%%%%%%%%%%%%%%%%%%%%%%%

\section{Topological preliminaries}\label{pre}
In this section we collect the topological facts needed to prove Theorem \ref{ti3} and its consequences. The results presented here are often formulated in greater generality than needed in the subsequent sections, as they may be of independent interest.

Throughout this paper, all topological spaces under consideration are homotopy equivalent to CW complexes.
\subsection{Finite homotopy/homological type} 
\bd Let $X$ be a connected topological space. We say that:
\begin{enumerate}
\item  $X$ has {\it finite $k$-homotopy type} if $X$ is homotopy equivalent to a CW complex with finite $k$-skeleton. 
\item $X$ has {\it finite homotopy type} if $X$ is homotopy equivalent to a finite CW complex.
\item  $X$ has {\it finite $k$-homological type} if $H_i(X;\Z)$ is a finitely generated abelian group for any $i \leq k$.
\item $X$ has {\it finite homological type} if $H_i(X;\Z)$ is a finitely generated abelian group for any $i \geq 0$.
\end{enumerate}
\ed

\br\rm Note that if $X$ has finite $k$-homotopy type, then $X$ has finite $k$-homological type. However, the converse fails to be true.\er

%define the {\it homological finiteness length} of $X$ as:
%$$\varkappa(X)=\sup\{  i \ | \ H_i(X,\Z) \ {\rm is \ a \ finitely \ generated \ abelian \ group } \}.$$
%Clearly, $\varkappa(X)$ is a homotopy invariant of $X$.

% (homotopy equivalent to) a connected CW complex, and denote its skeleta by $\{X^{(k)}\}_{k\geq 0}$. The {\it finiteness length} of $X$ is classically defined as:
%$$\varkappa(X)=\sup\{  k \ | \ X^{(k)} \ {\rm is \ finite } \} \leq \infty.$$

%\section{Complex affine varieties}\label{pre}
\subsection{Integral Alexander modules}\label{AM}
Let $X$ be a connected topological space of finite homotopy type, and $\xi:\pi_1(X) \to \Z$ a non-zero homomorphism. Denote by $X^{\xi}$ the infinite cyclic cover of $X$ defined by $\ker(\xi)$. % {\color{red}There is some problem with the above definition. In general, $X^\xi$ may not be connected. So this $X^\xi$ should be defined differently from the one defined subsection 2.5.} 
Then $X^{\xi}$ is the disjoint union of $[\Z: \im(\xi)]$ homeomorphic connected components.
The group of covering transformations of the cover $X^{\xi}$ is isomorphic to $\Z$, so the singular chain (or cellular) groups $C_i(X^{\xi};\Z)$ and homology groups $H_i(X^{\xi};\Z)$ become modules over the group ring $\Z[\Z]\cong \Z[t^{\pm 1}]$.
\bd The $\Z[t^{\pm 1}]$-module $H_i(X^{\xi};\Z)$ is called the {\it $i$-th integral Alexander module} of the pair $(X,\xi)$.
\ed

For future reference, we note that $\xi$ can be regarded as an element in $H^1(X;\Z)$ via the canonical identification:
\be\label{canid}  \Hom(\pi_1(X),\Z) \cong  \Hom(H_1(X),\Z) \cong H^1(X;\Z) .\ee
Moreover, by the representability theorem for cohomology, any such class $\xi\in H^1(X;\Z)$ is represented by a homotopy class of maps $X\to S^1$. Let $f_{\xi}:X\to S^1$ be a fixed homotopy representative for $\xi$.

For any subspace $Z$ of $X$ with inclusion map $j:Z \hookrightarrow X$, we denote by $Z^{\xi}$ the infinite cyclic cover of $Z$ defined by the composite homomorphism 
$$\xi \circ j_*:\pi_1(Z) \overset{j_*}{\to} \pi_1(X) \overset{\xi}{\to} \Z.$$
Note that this homomorphism is induced by $f_{\xi}|_Z:Z\to S^1$.

\medskip

With the above notations, we have the following preliminary result:
\bp\label{pp1} Let $X$ be a connected topological space of finite homotopy type, and fix a non-zero homomorphism $\xi:\pi_1(X)\to \Z$. 
\begin{enumerate}
\item[(a)] Assume that $X$ is obtained up to homotopy from a subspace $Z$, by attaching cells of dimension $\geq n$ (with $n \geq 2$). Moreover, assume that the infinite cyclic cover $Z^{\xi}$ has finite $(n-1)$-homological type. Then the infinite cyclic cover $X^{\xi}$ of $X$ has finite $(n-1)$-homological type, i.e., the integral Alexander modules $H_i(X^{\xi};\Z)$ are finitely generated abelian groups for any $i < n$. 
\item[(b)] If $X$ is obtained from $Z$ by attaching cells of dimension exactly $n$ and $Z^{\xi}$ has finite homological type, then $H_i(X^{\xi};\Z)$ is a finitely generated abelian group for any $i \neq n$.
\end{enumerate}
\ep

\begin{proof} Without loss of generality, we may assume that $X$ and $Z$ are CW complexes. %(Indeed, one can work with CW complexes homotopy equivalent to $X$ and resp. $Z$, and any continuous map between CW complexes may be converted up to homotopy to the inclusion of a subcomplex).
Let $j:Z \hookrightarrow X$ denote as before the inclusion map. Since $X \setminus Z$ contains only cells of dimension $\geq n$, the pair $(X,Z)$ is $(n-1)$-connected. So  the homomorphism 
%$\pi_i(Z) \to \pi_i(X)$ induced by inclusion is an isomorphism for $i<n-1$, and $\pi_{n-1}(Z) \to \pi_{n-1}(X)$ is onto. In particular, 
$j_*: \pi_1(Z) \to \pi_1(X)$ induced by inclusion is onto if $n=2$ and it is an isomorphism if $n>2$.

\noindent{(a)} If $X$ is obtained from $Z$ by adding cells of dimension $\geq n$, the same is true for the infinite cyclic covers $X^{\xi}$ and $Z^{\xi}$ defined by $\xi$ and $\xi \circ j_*$, respectively. It follows that the induced group homomorphism
$$H_i(Z^{\xi};\Z) \to H_i(X^{\xi};\Z)$$
is an isomorphism for $i<n-1$ and it is onto for $i=n-1$. The result follows now from the fact that $Z^{\xi}$ has finite $(n-1)$-homological type. 

\noindent{(b)} If $X$ is obtained from $Z$ by attaching cells of dimension exactly $n$,   the same is true for the infinite cyclic covers $X^{\xi}$ and $Z^{\xi}$. From the homology long exact sequence of the pair $(X^{\xi},Z^{\xi})$, it follows that the induced group homomorphism
$$H_i(Z^{\xi};\Z) \to H_i(X^{\xi};\Z)$$
is an isomorphism for $i<n-1$ and $i>n$, and it is onto for $i=n-1$.
The claim follows since $Z^{\xi}$ has finite homological type.
\end{proof}

\br\label{rn1}\rm
As it can be easily seen from the above proof, the conclusions of Proposition \ref{pp1} still hold if we only assume that the infinite cyclic cover  $X^{\xi}$ is built (up to homotopy) from a subspace of finite homological type by attaching (possibly infinitely many) cells of dimension $\geq n$, and resp. $n$. This is in fact the case for smooth closed subvarieties of complex semi-abelian varieties (see Theorem \ref{ti1}). In particular, this observation yields a proof of Theorem \ref{ti3} from the Introduction.
\er

\br\rm When the integral Alexander module $H_i(X^\xi; \Z)$ is a torsion $\Z[t,t^{-1}]$-module, it is not necessarily finitely generated as a $\Z$-module. For example, as noted in \cite[section 2]{Mil}, the integral Alexander module of the complement $X$ in $S^3$ of the knot $5_2$ is $\Z[t,t^{-1}] /(2t^2 -3t+2)$, which is torsion over $\Z[t,t^{-1}]$, but not finitely generated over $\Z$. 
\er

%\medskip
We next give a concrete geometric situation when the assumptions of Proposition \ref{pp1} are fullfilled. We begin with the following simple observation.
\begin{lemma}\label{l1}
Let $f:X \to S^1$ be a locally trivial topological fibration with fiber $F$ of finite $k$-homological type, and let $\xi:=f_*:\pi_1(X)\to \pi_1(S^1)\cong \Z$. Then $X^{\xi}$ has finite $k$-homological type, i.e., $H_i(X^{\xi};\Z)$ is a finitely generated abelian group for any $i \leq k$. If, moreover, the fiber $F$ has finite homological type, then so does $X^{\xi}$.
\end{lemma}

\begin{proof} The infinite cyclic cover $X^{\xi}$  is homeomorphic to $F \times \R$, hence homotopy equivalent to $F$, with compatible $\Z$-actions. It follows that for any $i \geq 0$, we have an isomorphism of abelian groups (and $\Z[t^{\pm 1}]$-modules):
\be
H_i(X^{\xi};\Z) \cong H_i(F;\Z),
\ee
which yields the claim.
%Since $F$ is of finite $k$-homological type, it follows that $X^{\xi}$ also has finite $k$-homological type.
\end{proof}

As an immediate consequence of Lemma \ref{l1} and Proposition \ref{pp1}, we have the following:
\bp\label{pp2}
Let $X$ be a connected topological space of finite homotopy type, and fix a homomorphism $\xi:\pi_1(X)\to \Z$ with a homotopy representative $f_{\xi}:X \to S^1$. 
\begin{enumerate}
\item[(a)] Assume that $X$ is obtained from a subspace $Z$, by attaching cells of dimension $\geq n$ (with $n \geq 2$). Moreover, assume that $f_{\xi}|_Z:Z \to S^1$ is a locally trivial topological fibration whose fiber $F$ has finite $(n-1)$-homological type. 
Then the infinite cyclic cover $X^{\xi}$ of $X$ has finite $(n-1)$-homological type.
\item[(b)] If $X$ is obtained from $Z$ by attaching cells of dimension exactly $n$, and the fiber of the fibration $f_{\xi}|_Z:Z \to S^1$ has finite homological type, then $H_i(X^{\xi};\Z)$ is a finitely generated abelian group for any $i \neq n$.
\end{enumerate}
%, for any $i < n$, the integral Alexander modules $H_i(X^{\xi};\Z)$ of the infinite cyclic cover $X^{\xi}$ are finitely generated abelian groups. 
\ep
%\begin{proof} By Lemma \ref{l1}, the infinite cyclic cover $X_{\xi}$ defined by  $\xi \circ j_*$ has  finite $(n-1)$-homological type.  The claim follows now from Proposition \ref{pp1}.\end{proof}

\bex\label{ex}\rm
As main examples of topological spaces satisfying the assumptions of Proposition \ref{pp2}, we mention the following:
\begin{enumerate}
\item complements in $\C^n$ of hypersurfaces in general position (or regular) at infinity, e.g., see \cite{Ma06,DL06}; here, if $X=\C^n \setminus \{ f=0\}$, we can take $\xi=f_*$, or any other positive homomorphism; the fibration structure corresponds in this case to the complement of the {\it link at infinity}.
\item complements of essential hyperplane arrangements in $\C^n$, with $\xi$ a positive (rank $1$) homomorphism, e.g., see \cite{Di02,KP15}; the fibration structure is given in this case by the {\it boundary manifold} of the arrangement. 
\item smooth connected very affine varieties, see Theorem \ref{ti2}; note that essential hyperplane arrangement complements in $\C^n$, as well as complements of toric hyperplane arrangements (e.g., studied in \cite{dCP1,dCP2}) are examples of smooth very affine varieties.
\end{enumerate}
Moreover, in all these cases, the space $X$ under cosideration is built from a fibration over the circle $S^1$ by adding $|\chi(X)|$ cells of middle (real) dimension $n$.
\eex

%%%%%%%%%%%%%%%%%%%%%%%%%%%%%%
%%%%%%%%%%%%%%%%%%%%%%%%%%%%%%

\subsection{Signed Euler characteristic property}

\bp\label{se} Let $X$ be a connected topological space of finite homotopy type, and fix a homomorphism $\xi:\pi_1(X)\to \Z$. Assume that $X$ is obtained up to homotopy from a subspace $Z$, by attaching $\ell$ cells of dimension exactly $n$. Moreover, assume that the infinite cyclic cover $Z^{\xi}$ has finite homological type. Then \be(-1)^n \chi(X) = \ell \geq 0.\ee
\ep

\begin{proof} Let $H_i(X^{\xi};\C)$ denote the $i$-th complex Alexander module of the pair $(X,\xi)$. It follows from Proposition \ref{pp1}(b) that $H_i(X^{\xi};\C)$ is a torsion $\C[t,t^{-1}]$-module, for all $i \neq n$. Since the $\C[t,t^{-1}]$-ranks of the complex Alexander modules can be used to compute the Euler characteristic of $X$, we then have that:
$$\chi(X)=(-1)^n \rk_{\C[t,t^{-1}]} H_n(X^{\xi};\C),$$ which yields the signed Euler characteristic property. Similarly, it follows by our assumption that the complex Alexander modules $H_i(Z^{\xi};\C)$ are torsion over $\C[t,t^{-1}]$, for all $i \geq 0$. %, so $\chi(Z)=0$.
The fact that $\rk_{\C[t,t^{-1}]} H_n(X^{\xi};\C)=\ell$ follows now from the homology long exact sequence of the pair $(X^{\xi},Z^{\xi})$, by noting that $H_n(X^{\xi},Z^{\xi};\C)\cong \C[t,t^{-1}]^{\oplus \ell}$, and $H_i(X^{\xi},Z^{\xi};\C)=0$ for all $i \neq n$.
\end{proof}

In view of Remark \ref{rn1} and Theorem \ref{ti1}, Proposition \ref{se} yields the signed Euler characteristic property (\ref{e1}) for smooth closed subvarieties of semi-abelian varieties, i.e., a proof of Corollary \ref{ci1}.

%%%%%%%%%%%%%%%%%%%%%%%%%%%%%%
%%%%%%%%%%%%%%%%%%%%%%%%%%%%%%

\subsection{Generic vanishing}

The following is an application of Propositions \ref{pp1} and \ref{pp2}.
\begin{proposition}\label{gv}
Let $X$ be a connected topological space of  finite homotopy type, with an epimorphism $\xi: \pi_1(X)\to \Z$. Assume that the associated infinite cyclic cover $X^\xi$ of $X$ has finite $k$-homological type. Given any $s\in \C^*$, denote by $L^{\xi}_s$ the rank-one $\C$-local system on $X$ whose monodromy representation is given by 
$$\pi_1(X)\xrightarrow{\xi} \Z\to \C^*,$$
where the homomorphism $ \Z\to \C^*$ is defined by $1\mapsto s$. Then for all but finitely many $s\in \C^*$, we have:
$$H_i(X, L^{\xi}_s)=0, \text{ for all } i\leq k.$$
\end{proposition}

\begin{proof}
Consider the Milnor long exact sequence (e.g., see \cite[Theorem 4.2]{DN04}):
$$\cdots \lra H_i(X^\xi;\C) \overset{t-s}{\lra} H_i(X^\xi;\C) \lra H_i(X, L^{\xi}_s) \lra H_{i-1}(X^\xi;\C) \lra \cdots.$$
Under our assumptions, $H_i(X^\xi;\C)$ is a finitely generated torsion $\C[t,t^{-1}]$-module for any $i \leq k$.  Hence multiplication by $t-s$ on $H_i(X^\xi;\C)$ is an isomorphism (for $i \leq k$) provided that $s$ is not a root of the Alexander polynomial associated to $H_i(X^\xi;\C)$. The claim follows.
\end{proof}

In the context of Proposition \ref{pp1} (b) and in the notations of the previous Proposition, we also have the following:
\bp\label{gvb} Let $X$ be a connected topological space of finite homotopy type, with an epimorphism $\xi: \pi_1(X)\to \Z$.  Assume that $X$ is obtained up to homotopy from a subspace $Z$, by attaching cells of dimension exactly $n$. Moreover, assume that the infinite cyclic cover $Z^{\xi}$ has finite homological type. Then for all but finitely many $s\in \C^*$, we have:
$$H_i(X, L^{\xi}_s)=0, \text{ for all } i \neq n.$$
\ep

\begin{proof} Proposition \ref{pp1}(b) yields that the complex Alexander modules $H_i(X^{\xi};\C)$ are torsion $\C[t,t^{-1}]$-modules, for all $i \neq n$. Let $\Delta^{\xi}_i(t)$ denote the associated Alexander polynomial. We also let $\Delta^{\xi}_n(t)$ be the Alexander polynomial associated to the torsion part of the finitely generated $\C[t,t^{-1}]$-module $H_n(X^{\xi};\C)$.  

By using the above Milnor long exact sequence, it follows as in \cite[Theorem 4.2]{DN04}, that for any $i \geq 0$, we have:
$$\dim H_i(X, L^{\xi}_s)=\rk_{\C[t,t^{-1}]} H_i(X^\xi;\C) + N(s,i)+N(s,i-1),$$ 
where $N(s,i)$ is the number of direct summands in the $(t-s)$-torsion of the $\C[t,t^{-1}]$-module $H_i(X^\xi;\C)$. So, as long as $s \in \C^*$ is chosen such that $H_i(X^{\xi};\C)$ has no $(t-s)$-torsion for any $i \geq 0$ (or, equivalently, $s$ is not a root of any of the Alexander polynomials $\Delta^{\xi}_i(t)$, for all $i \geq 0$), we have that $H_i(X, L^{\xi}_s)=0, \text{ for all } i \neq n$ and $\dim H_n(X, L^{\xi}_s)=\rk_{\C[t,t^{-1}]} H_n(X^\xi;\C)=(-1)^n \chi(X)$.
\end{proof}

\br\label{rn2}\rm
As in Remark \ref{rn1}, the generic vanishing of Proposition \ref{gvb} still holds if we only assume that $X^\xi$ is built (up to homotopy) from a subspace of finite homological type by adding (possibly infinitely many) cells of dimension $n$. In particular, in view of Theorem \ref{ti1}, this yields a proof of Corollary \ref{ci2} from the Introduction.
\er

%%%%%%%%%%%%%%%%%%%%%%%%%%%%%%%
%%%%%%%%%%%%%%%%%%%%%%%%%%%%%%%

\subsection{Novikov-Betti and torsion numbers}\label{NB}
Finiteness properties of infinite cyclic covers translate into vanishing of the corresponding Novikov-Betti and resp. Novikov-torsion numbers. For the convenience of the reader, we include here the relevant definitions.

Let $X$ be a connected topological space of finite homotopy type, and fix $\xi \in H^1(X;\R)$. Via the canonical identification 
$$H^1(X;\R)\cong \Hom(H_1(X),\R) \cong \Hom(\pi_1(X),\R)$$
any real cohomology class  $\xi \in H^1(X;\R)$ 
determines a homomorphism $\xi\colon \pi_1(X) \to \R$, where $\R$ is considered as a group with the usual addition. (We use here the same notation for the cohomology class and the corresponding homomorphism on the fundamental group; it should be clear from the context which one is referred to.) Let $\ga_{\xi}$ denote the image of the homomorphism $\xi$. Then $\xi$ factors as $\xi=i_{\xi} \circ \eta_{\xi}$, with $\eta_{\xi}:\pi_1(X) \twoheadrightarrow \ga_{\xi}$ an epimorphism and $i_{\xi} : \ga_{\xi} \hookrightarrow \R$ a monomorphism.
The group $\ga_{\xi}$ is a finitely generated free abelian group, say $\ga_{\xi} \cong \Z^r$, $r \geq 0$. The integer $r$ is called the {\it rank} of $\xi$. Note that if $\xi \in H^1(X;\Z)$ is an integral cohomology class, then it has rank $1$. The set of rank $1$ classes is dense in $H^1(X;\R)$.

Let $$X^{\xi} \lra X$$ be the covering of $X$ defined by $\ker (\xi)$. The group of covering transformations of $X^{\xi}$ is isomorphic to $\ga_{\xi}$, so the cellular groups $C_i(X^{\xi};\Z)$, as well as the homology groups $H_i(X^{\xi};\Z)$, become (finitely generated) left modules over the group ring $\Z\ga_{\xi} \cong \Z\tbigpm$. Let $Q_{\xi}$ denote the field of fractions of $\Z\ga_{\xi}$.
\bd The {\it $i$-th Novikov-Betti number} $b_i(X,\xi)$ of the pair $(X,\xi)$ is the $\Z\ga_{\xi}$-rank of the module $H_i(X^{\xi};\Z)$, i.e., the dimension of the $Q_{\xi}$-vector space $Q_{\xi} \otimes_{\Z\ga_{\xi}} H_i(X^{\xi};\Z)$. 
\ed

\br\rm Note that if $\xi\in H^1(X;\Z)$ is an integral cohomology class corresponding to an epimorphism $\xi:\pi_1(X) \to \Z$, then  the $i$-th Novikov-Betti number $b_i(X,\xi)$  is just the $\Z[t^{\pm 1}]$-rank of the integral Alexander module $H_i(X^{\xi};\Z)$.
\er

Alternatively, following \cite{F04}, one can define the Novikov-Betti numbers of $(X,\xi)$ as ranks over the {\it rational Novikov ring} of $\ga_{\xi}$, which is a certain localization of $\Z\ga_{\xi}$. In more detail, a Laurent polynomial $p=\sum_{\g} n_{\g}\g \in \Z\ga_{\xi}$ is called monic if $n_e=1$ and if for all $\gamma\ne e$ with $n_\gamma\ne 0$ we have $i_\xi(\gamma)<0$. 
Here $e$ is the identity element in the group $\ga_{\xi}$, hence also the multiplicative identity element in the ring $\Z\ga_{\xi}.$  Denote by $S_{\xi}$ the set of all monic polynomials in $\Z\ga_{\xi}$. 
 \bd The {\it rational Novikov ring} $\RC\ga_{\xi}$ of $\ga_{\xi}$ is defined as the localization of $\Z\ga_{\xi}$ at the multiplicative set $S_{\xi}$, that is, 
$$\RC\ga_{\xi}:=S_{\xi}^{-1}\Z\ga_{\xi}.$$
\ed 
The advantage of working over the rational Novikov ring is that, as shown in \cite[Lemma~1.15]{F04}, the ring $\RC\ga_{\xi}$ is a PID. 
Next note that 
\be
S_{\xi}^{-1}C_*(X^{\xi}):=S_{\xi}^{-1}\Z\ga_{\xi} \otimes_{\Z\ga_{\xi}} C_*(X^{\xi}) \cong \RC\ga_{\xi} \otimes_{\Z\pi_1(X)} C_*(\wti{X}),
\ee
where $\wti{X}$ denotes the universal cover of $X$. 
Hence, by the exactness of the localization functor, one has an isomorphism of $\RC\ga_{\xi}$-modules:
\be
H_*(X;\RC\ga_{\xi})\cong S_{\xi}^{-1}H_*(X^{\xi};\Z).
\ee
It follows that the $i$-th Novikov-Betti number of $(X,\xi)$ equals
\be b_i(X,\xi)=\rk_{\RC\ga_{\xi}}H_i(X;\RC\ga_{\xi}).
\ee 

\bd The {\it $i$-th Novikov torsion number} $q_i(X,\xi)$ is defined as the minimal number of generators of the torsion submodule of the $\RC\ga_{\xi}$-module $H_i(X;\RC\ga_{\xi})$.
\ed

The following result is well-known, e.g., see \cite[Chapter~1]{F04}:
\bp\label{pr}  For any $\xi \in H^1(X;\R)$, the Euler-Poincar\'e characteristic of $X$ is computed as: $$\chi(X)=\sum_{i \geq 0} (-1)^i \cdot b_i(X,\xi).$$
\ep
We also recall here the following simple but useful fact, see \cite[Proposition~1.35]{F04}:

\bp\label{pfg} Suppose that $\xi \in H^1(X;\R)$, $\xi \neq 0$, and the $i$-th homology $H_i(X^{\xi};\Z)$ is finitely generated as an abelian group. Then $b_i(X,\xi)=q_i(X,\xi)=0$. 
\ep

In particular, this result applies in the context of Propositions \ref{pp1} and \ref{pp2}. Namely, Proposition \ref{pp1} yields the following result (compare also with \cite{FM16, KP15} for special cases):

\begin{corollary}\label{cc1}
Let $X$ be a connected topological space of finite homotopy type,  and fix a (non-zero) epimorphism $\xi:\pi_1(X)\to \Z$. 
\begin{itemize}
\item[(a)] Assume that $X$ is obtained up to homotopy from a subspace $Z$, by attaching cells of dimension $\geq n$ (with $n \geq 2$). Moreover, assume that  the infinite cyclic cover $Z^{\xi}$ defined by the kernel of the composition $\pi_1(Z)\to \pi_1(X)\overset{\xi}{\to} \Z$ has finite $(n-1)$-homological type. Then, for any $i < n$, we have that $b_i(X,\xi)=0$ and $q_i(X,\xi)=0$.  
\item[(b)] If $X$ is obtained from $Z$ by attaching cells of dimension exactly $n$ and $Z^{\xi}$ has finite homological type, then $b_i(X,\xi)=q_i(X,\xi)=0$ for any $i \neq n$, and $b_n(X,\xi)=(-1)^n\chi(X)$.
\end{itemize}
\end{corollary}

Note that Theorem \ref{ti3} together with Propositions \ref{pr} and \ref{pfg} yields now a proof of Corollary \ref{ci3}.

%%%%%%%%%%%%%%%%%%%%%%%%%%%%%%
%%%%%%%%%%%%%%%%%%%%%%%%%%%%%%

\subsection{$L^2$-Betti numbers}\label{l2}
Novikov-Betti numbers can be regarded as special cases of $L^2$-Betti numbers. In fact, the following identification holds:
\[ b_i(X;\xi)=b_i^{(2)}(X,\xi\colon \pi_1(X) \to \im(\xi)).\]
However, Novikov torsion numbers do not have such interpretation in terms of $L^2$-invariants.
For completeness, we include here the definition of $L^2$-Betti numbers. %, and prove the corresponding vanishing/finiteness result in the setup of Theorem \ref{t2}.

\medskip

Let $\G$ be a countable group, with $$l^2(\G):=\{ f:\G\to \C \, | \, \sum_{g\in \G}  |f(g)|^2<\infty \}$$
the Hilbert space of square-summable functions on $\Gamma$. Then $\G$ acts on $l^2(\G)$ by right multiplication, so $\C[\G]$ can be regarded as a subset of $\BB(l^2(\G))$, the set of bounded operators on $l^2(\G)$.
The {\it von Neumann algebra} $\ng$ of $\G$ is defined as the closure of $\C[\G]\subset \BB(l^2(\G))$ with respect to pointwise convergence in
$\BB(l^2(\G))$. Any $\ng$--module $\MC$ has a dimension $\dim_{\ng}(\MC)\in \R_{\geq 0}\cup \{\infty\}$. We refer to \cite[Def.6.50]{Lu02} for details.

\medskip

\bd To any connected topological space $X$ of finite homotopy type and a group homomorphism $\alpha:\pi_1(X) \to \Gamma$, we associate {\it $L^2$-Betti numbers} 
$$b_i^{(2)}(X,\alpha):={\dim}_{\mathcal{N}(\Gamma)} H_i(X;\mathcal{N}(\Gamma)):={\dim}_{\mathcal{N}(\Gamma)} 
H_i \big(C_*(X^{\G}) \otimes_{\mathbb{Z}\Gamma} \mathcal{N}(\Gamma) \big) \in [0,\infty],$$ where $X^{\G}$ is the regular covering of $X$ defined by $\ker(\alpha)$, and $C_*(X^{\Gamma})$ is the singular chain complex of $X^{\Gamma}$ with (right) $\Gamma$-action given by covering translations. If $X$ is a CW complex, the cellular chain complex of $X^{\Gamma}$ can be used in the above definition of $L^2$-Betti numbers. 
\ed
\br\rm
Since 
\[ b_i^{(2)}(X,\a: \pi_1(X)\to \G)=b_i^{(2)}(X,\a: \pi_1(X)\to \im(\a)),\]
we can assume without any loss of generality that $\a$ is an epimorphism.\er

The following properties of $L^2$-Betti numbers are standard, we refer to \cite{Lu02} for details:
\begin{lemma}\label{l2}
For any connected topological space $X$ and group homomorphism $\alpha:\pi_1(X) \to \Gamma$,
\begin{itemize}
\item[(a)] $b_i^{(2)}(X,\alpha)$ is a homotopy invariant of the pair $(X,\alpha)$.
\item[(b)] $b_0^{(2)}(X,\a)=0$ if $\G$ is infinite and $b_0^{(2)}(X,\a)=\frac{1}{|\G|}$ if $\G$ is finite.
\item[(c)] if $X$ is of finite homotopy type, then 
$$\chi(X)=\sum_{i\geq 0} (-1)^i b_i^{(2)}(X,\alpha).$$
\end{itemize}
\end{lemma}

\br\rm Since all spaces we work with are homotopy equivalent to CW complexes, 
%by part (a) of the above Lemma 
we can assume, without loss of generality, that all spaces in this section are CW-complexes.
\er

Let us now assume that $X$ is a finite CW complex obtained from a CW (sub)-complex $Z$ by attaching cells of dimension $\geq n$. %satisfies the assumptions of Propositions \ref{pp1} or \ref{pp2}. 
Let $\a_Z$ be the composition
$$\a_Z:\pi_1(Z) \overset{j_*}{\lra} \pi_1(X) \overset{\a}{\lra} \G,$$  and let $Z^{\Gamma}$ be the corresponding $\Gamma$-cover of $Z$.
Let $b_i^{(2)}(Z,\a_Z)$ denote the $L^2$-Betti numbers of the pair $(Z,\a_Z)$. Then the following holds:
\bp\label{pp3} Assume that $X$ is a finite CW complex obtained from a CW (sub)-complex $Z$ by attaching cells of dimension $\geq n$. Then for any homomorphism $\a:\pi_1(X) \to \G$ we have:  
\[ b_i^{(2)}(X,\a) = b_i^{(2)}(Z,\a_Z), \ {\rm for} \ i<n-1, \]
and 
\[ b_{n-1}^{(2)}(X,\a) \leq b_{n-1}^{(2)}(Z,\a_Z).\]
If, moreover, $X$ is obtained from $Z$ by attaching cells of dimension exactly $n$, then we also have the equalities
\[ b_i^{(2)}(X,\a) = b_i^{(2)}(Z,\a_Z), \ {\rm for} \ i>n. \]
\ep
\begin{proof}
Since $X$ is obtained from $Z$ by adding cells of dimension $\geq n$, the same is true for their corresponding $\G$-coverings $X^{\G}$ and $Z^{\G}$, respectively. In particular, $C_i(Z^{\G})=C_i(X^{\G})$ for $i<n$. It follows that $H_i(Z;\mathcal{N}({\G}))\cong H_i(X;\mathcal{N}({\G}))$ for $i<n-1$, and the induced homomorphism $$H_{n-1}(Z;\mathcal{N}({\G}))\lra H_{n-1}(X;\mathcal{N}({\G}))$$ is surjective. The first claim follows now from  \cite[Thm.6.7]{Lu02}.

If, moreover, $X$ is obtained from $Z$ by attaching cells of dimension exactly $n$, then $C_i(Z,\mathcal{N}({\G}))=C_i(X,\mathcal{N}({\G}))$ for any $i \neq n$, and $C_n(X,\mathcal{N}({\G}))=C_n(Z,\mathcal{N}({\G})) \oplus P$, where $P$ is the free $\mathcal{N}({\G})$-module generated by the extra $n$-cells of $X$. Here, we use the notation $C_i(X,\mathcal{N}({\G})):=C_i(X^{\G}) \otimes_{\mathbb{Z}\Gamma} \mathcal{N}(\Gamma)$.
An easy diagram chase then yields the isomorphisms $H_i(Z;\mathcal{N}({\G}))\cong H_i(X;\mathcal{N}({\G}))$ for all $i>n$. This completes the proof.
\end{proof}

%\medskip

The $L^2$--Betti numbers provide obstructions for a space to fiber over a circle. More precisely, by \cite[Thm.1.39]{Lu02}, we have the following:

\begin{lemma}\label{l4} 
Let $X$ be a CW complex of finite homotopy type, and $f:X\to S^1$ a fibration with connected fiber $F$ of finite homotopy type. Assume that the epimorphism $f_*:\pi_1(X) \to \pi_1(S^1)=\Z$ admits a  factorization $\pi_1(X)\overset{\a}{\to} \G \overset{\beta}{\to} \Z$, with $\a$ and $\beta$ epimorphisms . Then $$b_i^{(2)}(X,\a)=0 \ , \ \ {\rm for \ all} \ i \geq 0.$$
\end{lemma}

\bd An epimorphism $\a:\pi_1(X) \to \G$ is called {\it $\xi$-admissible} if $\xi$ factors through $\a$.
\ed

We can now prove the following vanishing result for the $L^2$-Betti numbers of a space $X$ satisfying the assumptions of Proposition \ref{pp2} (compare also with \cite{Ma14} for a more concrete situation). Recall that for a homomorphism $\xi:\pi_1(X)\to \Z$ (regarded as an element in $H^1(X,\Z)$), we have fixed a homotopy representative $f_{\xi}:X\to S^1$.

\bt\label{th2}
Let $X$ be a finite connected CW complex, and fix an epimorphism $\xi:\pi_1(X)\to \Z$ with a homotopy representative $f_{\xi}:X \to S^1$. Assume that $X$ is obtained from a CW (sub-)complex $Z$, by attaching cells of dimension $\geq n$ (with $n \geq 2$). Moreover, assume that $f_{\xi}|_Z:Z \to S^1$ is a locally trivial topological fibration with fiber $F$ of finite homotopy type.. Then, for any $\xi$-admissible epimorphism $\a:\pi_1(X) \to \G$, we have that  $b_i^{(2)}(X,\a) =0$ for all $i \leq n-1$. If, moreover, $X$ is obtained from $Z$ by attaching cells of dimension exactly $n$, then $b_i^{(2)}(X,\a) =0$ for all $i \neq n$, and $b_n^{(2)}(X,\a) =(-1)^n\chi(X)$.
\et

\begin{proof}
Since, as shown in the proof of Proposition \ref{pp1}, $j_*:\pi_1(Z)\to \pi_1(X)$ is onto, it follows that $f_{\xi}|_Z$ induces an epimorphism on fundamental groups, so its fiber is connected. It also follows 
 that $\a_Z:\pi_1(Z)\to \G$ is an epimorphism, which is clearly admissible for $\xi \circ j_*=(f_{\xi}|_Z)_*$.
By Lemma \ref{l4} applied to the pair $(Z,\a_Z)$, we have that $b_i^{(2)}(Z,\a_Z)=0$ for all $i \geq 0$.
The desired result follows now from Proposition \ref{pp3}, together with Lemma \ref{l2}(c).\end{proof}

\br\label{l2r}\rm
The above result proves Corollary \ref{L2} from the Introduction (see also Remark \ref{rla})), and  gives another topological interpretation of the signed Euler characteristic property (\ref{e1}) for smooth closed subvarieties of semi-abelian varieties. (The proof of  Corollary \ref{L2} is trivial for $\dim(X)=1$, since $X$ is homotopy equivalent to a $1$-dimensional CW-complex in this case.)
\er

%%%%%%%%%%%%%%%%%%%%%%%%%%%%%%
%%%%%%%%%%%%%%%%%%%%%%%%%%%%%%

\section{Non-proper Morse theory}\label{np}

In this section, we recall the main aspects of the non-proper Morse theory due to Palais-Smale \cite{PS64}, as well as its  circle-valued extension. Non-proper Morse theory is the main tool needed in this paper to study the topology of subvarieties of semi-abelian varieties.

\subsection{Non-proper Morse theory after Palais-Smale}

Let $f: M\to \R$ be a real-valued smooth function on a smooth manifold $M$. For any real numbers $a<b$, we define $f^{a,b}:=f^{-1}\big([a, b]\big)$ and $f^a:=f^{-1}\big((-\infty, a]\big)$. The following non-proper Morse theory result is due to Palais-Smale: 
\begin{theorem}[\cite{PS64}]\label{nonproper}
Let $M$ be a complete Riemannian manifold, and let $f: M\to \R$ be a real-valued Morse function satisfying the following: 
\begin{description}
\item[Condition]\label{condition} If $S$ is a subset of $M$ on which $|f|$ is bounded but on which $||\nabla f||$ is not bounded away from zero, then there exists a critical point of $f$ in the closure of $S$. 
\end{description}
Then the following properties hold: 
\begin{enumerate}
\item For any real numbers $a<b$, there are finitely many critical points of $f$ in $f^{a,b}$.
\item Let $a, b$ be regular values of $f$. Suppose that there are $r$ critical points of $f$ in $f^{a, b}$ having index $d_1, \ldots, d_r$, respectively. Then $f^{b}$ has the homotopy type of $f^{a}$ with $r$ cells of dimensions $d_1, \ldots, d_r$ attached. 
\item If $c$ is a regular value of $f$, then $f$ has the structure of a fiber bundle in a small neighborhood of $c$. 
\end{enumerate}
\end{theorem}
\begin{proof}
Let us first discuss (3). Let $\epsilon>0$ be sufficiently small such that $f^{-1}\big([c-\epsilon, c+\epsilon]\big)$ does not contain any critical point of $f$. By the condition on $f$, $||\nabla f||$ is bounded away from zero. Using a partition of unity, we can construct a smooth function on $M$ which is equal to $||\nabla f||$ on $f^{-1}\big([c-\epsilon, c+\epsilon]\big)$ and is bounded away from zero. Dividing the Riemannian metric on $M$ by the above function, we have a new complete Riemannian metric, since the function is bounded away from zero. Moreover, with respect to the new metric, $||\nabla f||=1$ on $f^{-1}\big([c-\epsilon, c+\epsilon]\big)$. Therefore, the gradient flow provides a trivialization of $f$ over $[c-\epsilon, c+\epsilon]$. Thus, $f$ has the structure of a fiber bundle over $(c-\epsilon, c+\epsilon)$. 

Properties (1) and (2) are stated without proof in \cite[Theorem 2, Page 166]{PS64}. Nevertheless, they follow from  standard Morse theory arguments. For instance, it is not hard to give a proof using the above reasoning together with the standard proper Morse theory. 
\end{proof}

%%%%%%%%%%%%%%%%%%%%%%%%%%%%%%%%%%%

\subsection{Circle-valued non-proper Morse theory}

In this section we present a circle-valued version of Theorem \ref{nonproper}.  

Let $f: M\to S^1$ be a circle-valued smooth function on a smooth manifold. Here we consider $S^1=\R/\Z$. For any real numbers $a, b$ with $a<b<a+1$, we define $f^{\overline{a},\overline{b}}:=f^{-1}\big(\overline{[a, b]}\big)$, where $\overline{[a, b]}$ is the image of the closed interval $[a, b]$ in $S^1$. 

\begin{theorem}\label{nonpropercircle}
Let $M$ be a complete Riemannian manifold, and let $f: M\to S^1$ be a circle-valued Morse function satisfying the following: 
\begin{description}
\item[Condition]\label{condition2} If $S$ is a subset of $M$ on which $||\nabla f||$ is not bounded away from zero, then there exists a critical point of $f$ in the closure of $S$. 
\end{description}
Then the following hold:
\begin{enumerate}
\item There exists finitely many critical points of $f$ in $M$;
\item Let $a, b$ be real numbers with $a<b<a+1$ such that $\overline{a}$ and $\overline{b}$ are regular values of $f$. Suppose that there are $r$ critical points of $f$ in $f^{\overline{a}, \overline{b}}$ of index $d_1, \ldots, d_r$, respectively. Then $f^{\overline{a}, \overline{b}}$ has the homotopy type of $f^{-1}(\overline{a})$ with $r$ cells of dimensions $d_1, \ldots, d_r$ attached. 
\item  If $\overline{c}$  is a regular value of $f$, then $f$ has the structure of a fiber bundle in a small neighborhood of $\overline{c}$ . 
\end{enumerate}
\end{theorem}
\begin{proof}
Let $\wti{f}: \wti{M}\to \R$ be the lifting of $f: M\to S^1$ to the infinite cyclic cover $\wti{M}$ corresponding to $f$. Then by lifting $S$ to the fundamental domain $\wti{f}^{-1}\big([a, a+1)\big)$, one can easily check that the Palais-Smale condition of Theorem \ref{nonproper} applies to $\wti{f}$. Thus (1) and (3) follow from the corresponding statements of Theorem \ref{nonproper}.

By Theorem \ref{nonproper} (3), there is a diffeomorphism between $\wti{f}^{-1}\big((a-\epsilon, a)\big)$ and $\wti{f}^{-1}(a)\times (a-\epsilon, a)$ for some small $\epsilon>0$. Thus, we can glue $\wti{f}^{-1}(a-\epsilon, +\infty)$ with $\tilde{f}^{-1}(a)\times (-\infty, a)$ via the above diffeomorphism and obtain a new manifold $\wti{M}'$ and a function $\wti{f}': \wti{M}'\to \R$. We can construct a complete Riemannian metric on $\wti{M}'$ as follows. Let $\rho: \R\to \R$ be a smooth non-decreasing function, such that $\rho=0$ on $(-\infty, a-\epsilon]$ and $\rho=1$ on $[a, +\infty)$. Define the metric on $\wti{M}'$ to be 
$$g_{\wti{M}'}=(1-\rho)\cdot g_{\wti{f}^{-1}(a)\times (-\infty, a)}+\rho\cdot g_{\wti{f}^{-1}(a-\epsilon, +\infty)},$$
which is clearly complete. Now, notice that $(\tilde{f}')^{b}$ is homotopy equivalent to $f^{\overline{a}, \overline{b}}$ and $(\tilde{f}')^{a}$ is homotopy equivalent to $f^{-1}(\overline{a})$. Thus (2) follows from Theorem \ref{nonproper} (2) applied to $(\wti{M}', \wti{f}')$. 
\end{proof}

%%%%%%%%%%%%%%%%%%%%%%%%%%%%%%%%%%%%%
%%%%%%%%%%%%%%%%%%%%%%%%%%%%%%%%%%%%%

\section{Topology of semi-abelian varieties. Statements of results}\label{mr}

In this section, we introduce the main constructions and results of this note. We defer some of the technical proofs for subsequent sections. We also prove Theorem \ref{ti1} (see Corollary \ref{topologyone}) and Theorem \ref{ti2} (see Corollary \ref{fibrb} and Corollary \ref{fibrbb}).

\subsection{Subvarieties of semi-abelian varieties}
A nice introduction of complex semi-abelian varieties is \cite[Chapter5]{NW14}. 
We recall here the following definition.

\begin{definition}
A commutative complex algebraic group $G$ is called a {\it semi-abelian variety} if there is a short exact sequence of complex algebraic groups
\begin{equation}
1\to T\to G\to A\to 1,
\end{equation}
where $T$ is an affine torus and $A$ is an abelian variety. 
\end{definition}

Let $X$ be an $n$-dimensional smooth connected closed subvariety of a semi-abelian variety $G$. Without loss of generality, we will assume that $X$ contains the origin $e$ of $G$, and we will always take $e$ as the base point of $G$ and $X$, respectively. Let $\Gamma$ be the space of left invariant holomorphic $1$-forms on $G$. For a left invariant $1$-form $\eta\in \Gamma$, we denote its restriction to $X$ by $\eta_X$. Suppose the class $[\Re(\eta)]\in H^1(G; \R)$ is integral, i.e., $[\Re(\eta)]$ is in the image of the natural map  $H^1(G; \Z)\to H^1(G; \R)$. Then $[\Re(\eta_X)]\in H^1(X; \R)$ is also integral. The map $$x\mapsto \int_e^x \Re(\eta_X)$$given by integration along paths  defines a multi-valued map from $X$ to $\R$. Since the ambiguity lies in the subgroup $\Z\subset \R$, we have a well-defined smooth map $$\int \Re(\eta_X): X\to S^1,$$ which we denote by $\Phi_\eta$. Notice that $\Phi_\eta$ is the restriction of $\int\Re(\eta): G\to S^1$, which we denote by $\Phi_{\eta, G}$. Moreover, $\Phi_{\eta, G}$ is a homomorphism of real Lie groups, because $\Re(\eta)$ is a left invariant $1$-form on $G$. 

In the above notations we have the following result, whose proof will be given in Section \ref{PS}.

\begin{theorem}\label{main}
Regarding $\Gamma$ as a complex affine space, there exists a non-empty Zariski subset $U$ of $\Gamma$ such that the following holds: if $\eta\in U$ with $[\Re(\eta)]\in H^1(G; \R)$ integral, then $\Phi_\eta: X\to S^1$ is a circle-valued Morse function which 
satisfies the the circle-valued Palais-Smale condition of Theorem \ref{nonpropercircle}. 
\end{theorem}
%{\color{magenta}Suppose $G$ is not compact, i.e. abelian variety. For an integral class $s\in H^1(G, \R)$, the set of left invariant holomorphic 1-forms $\eta$ on $G$ such that $[\Re \eta]=s$ form a positive dimensional real vector space (the dimension is equal to the dimension of the affine torus part of $G$). For example, on $\C^*$, the real part of the holomorphic form $\sqrt{-1}\frac{dz}{z}$ is exact. 

%It is true that for a generic real Lie group map $G\to S^1$, the circle-valued Palais-Smale condition holds. But we need some clarification about what "generic" means here. }

\br \rm 
The metric on $X$ is the restriction of any left invariant Riemannian metric on $G$. 
Since any left invariant Riemannian metric on a Lie group $G$ is complete, and any closed submanifold of a complete Riemannian manifold is also complete, it follows that $X$ is a complete Riemannian manifold.
Furthermore, since any two left invariant Riemannian metrics $g_1$ and $g_2$ on $G$ are equivalent, i.e., $\frac{1}{M}g_1\leq g_2\leq Mg_1$ for sufficiently large $M$, the statement of Theorem \ref{main} does not depend on the choice of the left invariant metric on $G$. 
%Furthermore, the metric on $X$ is {\it complete}, since $X$ is a closed submanifold of the complete Riemannian manifold $G$. Here we use the fact that any left invariant Riemannian metric on a Lie group $G$ is complete. % (because a fixed ball at the origin can be moved to any point in the Lie group). A closed sub-manifold of a complete Riemannian manifold is also complete (any path in the sub-manifold is also a path in the ambient manifold, so the distance becomes larger in the sub-manifold, and hence every closed ball is compact.)
\er

\begin{corollary}\label{fiber}
Suppose $\eta\in U$ is chosen as in the Theorem \ref{main} and assume $[\Re(\eta)]\in H^1(G; \R)$ is integral. Let $\wti{X}$ be the infinite cyclic cover of $X$ with respect to the map $\Phi_\eta: X\to S^1$, and let $\wti{\Phi}_\eta: \wti{X}\to \R$ be the lifting of $\Phi_\eta: X\to S^1$. Then for a regular value $a$ of $\wti\Phi_\eta$, $\wti{X}$ has the homotopy type of $\wti{\Phi}_\eta^{-1}(a)$ with possibly infinitely many $n$-cells attached. 
\end{corollary}

\begin{proof}
Since $\eta$ is a holomorphic 1-form, locally the Morse function $\Phi_\eta$ is equal to the real part of a holomorphic Morse function. By the holomorphic Morse lemma (see, e.g., \cite[Section 2.1.2]{Vo03}), the index of $\Phi_\eta$ at every critical point is equal to the complex dimension of $X$. The assertion follows now from Theorem \ref{main} and Theorem \ref{nonpropercircle}. 
\end{proof}

%\begin{corollary}
%For a generic\footnote{The word "generic" is with respect to the following Zariski topology on $Hom(G, S^1)$. Notice that the real linear map $\Re: T_e^{*(1,0)}G\to T_eG$ is an isomorphism. Thus, there is an isomorphism of real vector spaces between left invariant holomorphic 1-forms on $G$ and left invariant real 1-forms on $G$. Via this isomorphism, we can give the set of left invariant real 1-forms on $G$ a complex Zariski topology. Moreover, there is a bijection between $Hom(G, S^1)$ and the left invariant real 1-forms on $G$ whose cohomology classes are integral. Via this bijection, we define the topology on $Hom(G, S^1)$ to be the subspace topology of the Zariski topology on the left invariant real 1-forms. } real Lie group homomorphism $\phi: G\to S^1$, its restriction $\phi_X: X\to S^1$ satisfies the conclusion of the previous corollary. {\color{magenta} The statement of the corollary and the proof may need to be rewritten depending on future context. }
%\end{corollary}

In order to apply the topological vanishing results to the semi-abelian setting, we need $\wti{\Phi}_\eta^{-1}(a)$ to have finite homological or finite homotopy type. This is always the case when $G$ is an abelian variety, because then $\wti{\Phi}_\eta^{-1}(a)$ is a compact real analytic space. However, when $G$ is not compact, $\wti{\Phi}_\eta^{-1}(a)$ may not have finite homotopy type. For example, when $X=G=\C^*$ with coordinate $z$, and $\eta=2\pi\sqrt{-1}dz/z$, $\Phi_\eta: \C^*\to S^1$ is the composition of the following two maps
$$\C^*\to\R\to S^1,$$
where the first map is defined by $z\mapsto \log|z|$ and the second is taking the quotient by $\Z$. Clearly, the fibers of $\Phi_\eta: \C^*\to S^1$ have infinitely many connected components, so the same holds for $\wti{\Phi}_\eta$. 

\medskip

Let $T$ be the affine torus in the semi-abelian variety $G$ such that the quotient $G/T$ is an abelian variety. 
Denote the restriction of $\eta$ to $T$ by $\eta_T$. We have the following result whose proof will be given in Section \ref{mp}.

\begin{theorem}\label{finite}
Let $\Phi_\eta: X\to S^1$ be the map defined by $\eta\in \Gamma$ (not necessarily in $U$) with $[\Re(\eta)]\in H^1(G; \R)$ integral. If, furthermore, the cohomology class $[\eta_T]\in H^1(T; \C)$ is real, i.e., in the image of $H^1(T; \R)\to H^1(T; \C)$, then for any $\overline{a}\in S^1$, $\Phi_\eta^{-1}(\overline{a})$ has finite homotopy type. 
\end{theorem}

When $\dim T>0$, the real linear map $$\Gamma\lra H^1(G; \R)$$ given by $\eta\mapsto [\Re(\eta)]$ is surjective but not injective. In other words, given a group homomorphism $\xi: \pi_1(G)\to \Z$, there may exist many $\eta\in \Gamma$ such that the homomorphism induced by $\Phi_\eta: G\to S^1$ on fundamental groups coincides with $\xi$. However, if we restrict the above map to $\{\eta\in \Gamma| \ [\eta_T]\in H^1(T; \C) \text{ is real}\}$, then such $\eta$ is unique. More precisely, we have the following result, whose proof will be given in Section \ref{mp}. 

\begin{proposition}\label{representative}
For any class $a\in H^1(G; \R)$, there exists a unique $\eta\in \Gamma$ such that $[\Re(\eta)]=a\in H^1(G; \R)$ and $[\eta_T]\in H^1(T; \C)$ is real, i.e., in the image of the natural map $H^1(T; \R)\to H^1(T; \C)$. 
\end{proposition}

\begin{corollary}\label{topologyone}
Let $X$ be an $n$-dimensional closed connected smooth subvariety of a complex semi-abelian variety $G$. For a generic group homomorphism $\xi: \pi_1(X)\to \Z$, the corresponding infinite cyclic cover $X^\xi$ is homotopy equivalent to a finite CW-complex with possibly infinitely many $n$-cells attached. Moreover, if $\xi$ is surjective, then $X^\xi$ is connected.
\end{corollary}

\br \rm The meaning of ``generic'' in the above result is with respect to some Zariski topology on $\Hom(\pi_1(X), \Z)$. In the proof below, we will construct a complex vector space $\Gamma'$ and an injective map $\Hom(\pi_1(X), \Z)\to \Gamma'$. There is a non-empty Zariski open subset $U'$ of $\Gamma'$ such that the statement holds for all $\xi: \pi_1(X)\to \Z$ whose image in $\Gamma'$ is contained in $U'$. 
\er

\begin{proof}
Let $a_X: X\to Alb(X)$ be the generalized Albanese map. By the universal property of the generalized Albanese map, the closed embedding $X\hookrightarrow G$ factors as $X\xrightarrow[]{a_X} Alb(X)\to G$. Since the composition $X\to G$ is a closed embedding, the first map $a_X: X\to Alb(X)$ must also be a closed embedding. Denote $Alb(X)$ by $G'$ and consider $X$ as a subvariety of $G'$ via the Albanese map. 

By definition, $(a_X)_*: H_1(X; \Z)\to H_1(G'; \Z)$ is surjective, and its kernel is equal to the torsion subgroup of $H_1(X; \Z)$. Thus, there is a canonical bijection between $\Hom(\pi_1(X), \Z)$ and $\Hom(\pi_1(G'), \Z)$. Since $H_1(X; \Z)$ is the abelianization of $\pi_1(X)$, and by the universal coefficient theorem, we have
$$\Hom(\pi_1(G'), \Z)\cong \Hom(H_1(G'; \Z), \Z)\cong H^1(G'; \Z).$$
Let $\Gamma'$ be the space of left invariant holomorphic 1-forms on $G'$. By Proposition \ref{representative}, there is a natural map $H^1(G'; \Z)\to \Gamma'$ defined by $a\mapsto \eta$. Composing the maps $\Hom(\pi_1(X), \Z)\to \Hom(\pi_1(G'), \Z)$, $\Hom(\pi_1(G'), \Z)\cong H^1(G'; \Z)$, and, resp.,  $H^1(G'; \Z)\to \Gamma'$, we obtain a map $\Hom(\pi_1(X), \Z)\to \Gamma'$, which induces an isomorphism between $\Hom(\pi_1(X), \Z)$ and a discrete subgroup of $\Gamma'$. 

Let $U'$ be the subset of $\Gamma'$ chosen as in Theorem \ref{main}. Suppose the element of $\Gamma'$  corresponding to $\xi: \pi_1(X)\to \Z$ is contained in $U'$. By Theorem \ref{main}, Theorem \ref{nonpropercircle} and Theorem \ref{finite}, it follows that $X^\xi$ is homotopy equivalent to a finite CW-complex with possibly infinitely many $n$-cells attached. 

One can easily see that $\pi_0(X^\xi)$ is isomorphic to the cokernel of $\xi: \pi_1(X)\to \Z$. Thus, if $\xi$ is surjective, then $X^\xi$ is connected. 
\end{proof}

\begin{remark}\label{ro}\rm
In the preceding corollary, the Zariski topology on $\Hom(\pi_1(X), \Z)$ is coarser than the standard Zariski topology on $\Hom(\pi_1(X), \Z)\cong \Z^{b_1(X)}$. Evidently, the subset of surjective maps in $\Hom(\pi_1(X), \Z)$ is dense with respect to the standard Zariski topology on $\Hom(\pi_1(X), \Z)\cong \Z^{b_1(X)}$. Thus, the set 
$$\{\xi\in \Hom(\pi_1(X), \Z)| \ \xi \text{ is surjective and generic}\}$$
is non-empty. 
\end{remark}

%%%%%%%%%%%%%%%%%%%%%%%%%%%%%%%%%%%%
%%%%%%%%%%%%%%%%%%%%%%%%%%%%%%%%%%%%

\subsection{Subvarieties of affine torus}
In the case when $G$ is an affine torus, the topology of smooth subvarieties of $G$ can be made even more explicit.

Under the notations of Theorem \ref{main}, suppose moreover that $G$ is an affine torus. %Then there is a canonical isomorphism 
%$$H^0(\overline{G}, \Omega^1_{\overline{G}}(\log D))\cong H^1(G, \C),$$
%since $H^1(G, \C)$ is of pure (1,1) type. 
Let $\eta\in U$ such that $[\eta]\in H^1(G; \C)$ is integral, i.e., in the image of the natural map $H^1(G; \Z)\to H^1(G; \C)$. Via the isomorphism $\C/\Z\cong \C^*$, the integration along paths starting at $e$ $$x\mapsto \int_e^x\eta$$ defines a map $\int \eta: G\to\C^*$, which is a homomorphism of algebraic groups. Denote the composition of $\int \eta$ and $\log |\cdot|: \C^*\to \R$ by $\Psi_{\eta, G}:=\log |\int \eta|: G\to \R$.  Let $\Psi_\eta$ be the restriction of $\Psi_{\eta, G}$ to $X$. The following result will be proved in Section \ref{PS}.

\begin{theorem}\label{veryaffine}  Suppose $\eta\in U$ is chosen as in the Theorem \ref{main} and assume  $[\eta]\in H^1(G; \C)$ is integral. Then the function $\Psi_\eta: X\to \R$ is a Morse function which satisfies the the Palais-Smale condition of Theorem \ref{nonproper}. 
\end{theorem}

\begin{corollary}\label{fibr}  Suppose $\eta\in U$ is chosen as in the Theorem \ref{main} and assume $[\eta]\in H^1(G; \C)$ is integral. 
Then for a sufficiently negative real number $a$, $\Psi_\eta^{-1}\big((-\infty, a]\big)$ is a fiber bundle over $S^1$, whose fiber is homotopy equivalent to a finite CW-complex.  Moreover, $X$ has the homotopy type of $\Psi_\eta^{-1}\big((-\infty, a]\big)$ with finitely many $n$-cells attached. 
\end{corollary}

\begin{proof}
The map $\int\eta_X: X\to \C^*$ is a morphism of smooth complex algebraic varieties.  Hence, by the generic fibration theorem (see \cite[Corollary 5.1]{Ve76}), for sufficiently small $\epsilon>0$, the map $\int\eta_X$ is a fiber bundle over the punctured disc $D_\epsilon(0)\setminus \{0\}$, whose fibers are algebraic varieties. Thus, the first part of the corollary follows. 

As a morphism of algebraic varieties, $\int\eta_X: X\to \C^*$ has only finitely many critical values. The critical values of $\Psi_\eta$ are the logarithms of the absolute values of the critical values of $\int\eta_X$. Hence $\Psi_\eta$ has only finitely many critical values. In the proof of Theorem \ref{veryaffine}, we will show that $\int\eta_X: X\to \C^*$ is a holomorphic Morse function. A standard fact is that the absolute value of a non-vanishing holomorphic Morse function on a complex manifold of dimension $n$ is a real Morse function whose index at each critical point is $n$ (e.g., the argument in \cite{KP15}[Lemma 2.1] can be easily adapted to the general case). Since $\log : \R_{>0}\to \R$ is a smooth function with positive derivatives, $\Psi_{\eta}: X\to \R$ is a real-valued Morse function, which has only finitely many critical points, all of which are of index $n$. Thus, the second part of the corollary follows from Theorem \ref{veryaffine} and Theorem \ref{nonproper}. 
\end{proof}

\br\label{rnew}\rm
The fiber in the above statement is connected if the integral cohomology class $\eta$ is not divisible by any integer $>1$. (Indeed, under this assumption, the homomorphism induced by the bundle map on fundamental groups is onto, so the fiber is connected.) As in Remark \ref{ro}, it can be seen that there exists $\eta \in U$ so that $[\eta]$ (when regarded as an integral class) is not divisible by any integer $>1$. 
\er

\br\label{number}\rm
Obviously, the number of $n$-cells attached in the the statement of Corollary \ref{fibr} is equal to the number of critical points of $\Psi_{\eta}$. It follows from the arguments in the proof that the critical points of $\Psi_{\eta}$ are same as the degeneration points of $\eta_X$. Thus, the number of $n$-cells attached is equal to the number of degeneration points of $\eta_X$. 
\er

\begin{corollary}\label{fibrb}
Let $G$ be an affine torus, and let $X\subseteq G$ be a smooth connected closed subvariety of dimension $n$.  Then for any generic homomorphism $\xi: \pi_1(X)\to \Z$, which factors through $\pi_1(X) \to \pi_1(G)$, there exists a continuous map $F_\xi: X_0\to S^1$ defined on a subset $X_0$ of $X$ such that the following properties hold: 
\begin{enumerate}
\item $F_\xi$ is a fibration whose fiber is of finite homotopy type.
\item $X$ is homotopy equivalent to $X_0$ with finitely many $n$-cells attached.
\end{enumerate}  
\end{corollary}
\begin{proof}
The meaning of ``generic''  in the above statement is with respect to the standard Zariski topology on $\Hom(G, \C^*)\cong \Z^{\dim G}$. 

On the affine torus $G$, the map $\eta\mapsto [\eta]: \Gamma\to H^1(G; \C)$ is an isomorphism. Thus, each class $H^1(G; \Z)$ is represented by a unique left invariant holomorphic $1$-form. 
 By (\ref{canid}), each homomorphism $\xi: \pi_1(X) \to \Z$ corresponds to a unique integral cohomology class in $H^1(X, \Z)$. Choose $\xi$ such that the corresponding cohomology class is $\eta_X$, where $\eta$ is generic as in Theorem \ref{main}. Set $X_0=\Psi_\eta^{-1}\big((-\infty, a]\big)$ for sufficiently negative $a$, and let $F_\xi$ be the restriction of the composition $X\xrightarrow[]{\int \eta}\C^*\xrightarrow[]{\arg}S^1$ to $X_0$. The assertion follows now from the Corollary \ref{fibr}.
\end{proof}

\br\label{rla}\rm
Note that if $\xi$ in Corollary \ref{fibrb} is an epimorphism and $n >1$, the fiber of the fibration $F_{\xi}$ is connected. Indeed, as in the proof of Proposition \ref{pp1}, the homomorphism $\pi_1(X_0) \to \pi_1(X)$ induced by inclusion is in this case onto, so $F_{\xi}$ induces an epimorphism on fundamental groups. When $n=1$, $X$ is homotopy equivalent to a connected $1$-dimensional CW complex (since it is affine).
\er 

\br\rm In the case of a very affine manifold $X \subseteq G$ as in the above corollary, the homomorphism $\pi_1(X) \to \pi_1(G)$ induced by inclusion is non-trivial, provided that $\dim X>0$. In more detail, let $h$ be the restriction to $X$ of the polynomial map $\prod_{i=1}^m z_i : G \to \C^*$, where $(z_1, \cdots, z_m)$ are the coordinates of $G$. Without any loss of generality, we may assume that $h$ is dominant. (If not, the image of $h$ is just one point, say $b\in \C^*$, since $X$ is connected. Then $G_b:=\{\prod_{i=1}^m z_i=b\}$ is a complex  subtorus of $G$, of complex codimension one, which contains $X$. Replace $G$ by $G_b$. After finitely steps of replacing the torus containing $X$ by  codimension-one subtori, the map $h$ becomes dominant.) Then it is easy to see that the homomorphism $h_*:\pi_1(X)\to \pi_1(\C^*)$ induced by $h$ on  fundamental groups is non-trivial, provided that $\dim X>0$. Since, by definition, $h_*$ factors trough $\pi_1(X) \to \pi_1(G)$, it follows that the latter is also a non-trivial homomorphism.
\er

A more intrinsic version of the above Corollary \ref{fibrb} can be given as follows.

\begin{corollary}\label{fibrbb}
Let $X$ be an $n$-dimensional smooth connected closed subvariety of an affine torus $G$. Suppose the mixed Hodge structure of $H^1(X; \Z)$ is pure of type $(1,1)$, or equivalently, there exists a smooth compactification $\bar{X}$ of $X$ such that $b_1(\bar{X})=0$. Then for a generic homomorphism $\xi: \pi_1(X)\to \Z$, there exists a continuous map $F_\xi: X_0\to S^1$ defined on a subset $X_0$ of $X$ such that the following properties hold: 
\begin{enumerate}
\item $F_\xi$ is a fibration whose fiber is of finite homotopy type.
\item $X$ is homotopy equivalent to $X_0$ with finitely many $n$-cells attached.
\end{enumerate}
\end{corollary}

\begin{proof}
Note that $\Hom(\pi_1(X), \Z)$ is a free abelian group of rank $b_1(X)$. The meaning of  ``generic'' in the above statement is with respect to the standard Zariski topology on $\Hom(\pi_1(X), \Z)\cong \Z^{b_1(X)}$.

Let $a_X: X\to Alb(X)$ be the generalized Albanese map of $X$. Since $H^1(X; \Z)$ is of type $(1,1)$, $Alb(X)$ is an affine torus. By the same argument as in the proof of Corollary \ref{topologyone}, the Albanese map $a_X: X\to Alb(X)$ is a closed embedding.  Then the conclusion follows from the Corollary \ref{fibrb} with $G$ replaced by $Alb(X)$. 
\end{proof}

%\br\rm
%Note that $\Hom(\pi_1(X), \Z)$ is a free abelian group of rank $b_1(X)$. The meaning of  ``generic'' in the above statement is with respect to the standard Zariski topology on $\Hom(\pi_1(X), \Z)\cong \Z^{b_1(X)}$.
%\er

%%%%%%%%%%%%%%%%%%%%%%%%%%%%

\subsection{Proof of Theorem \ref{huhgen}}

\begin{proof}[Proof of Theorem \ref{huhgen}]
The statements (i) and (ii) follow easily from the arguments in the beginning of Section \ref{PS}. 

To prove (iii), let us first assume that $G$ is an affine torus and $[\Re(\eta)]\in H^1(G; \R)$ is integral. Then (iii) follows from Proposition \ref{se}, Corollary \ref{fibr} and Remark \ref{number}. Notice that  the number of critical points is constant on a non-empty Zariski open set of $\Gamma$ (see the arguments in the beginning of Section \ref{PS}). Since the subset $\{\eta\in\Gamma \ | \ [\Re(\eta)]\in H^1(G; \R) \textrm{ is integral}\}$ is clearly Zariski dense in $\Gamma$, the statement (iii) follows assuming $G$ is an affine torus.

When $G$ is a general semi-abelian variety, we can apply a similar argument, but using Corollary \ref{topologyone} and Theorem \ref{finite} instead of Corollary \ref{fibr}, while instead of Remark \ref{number} we need the parallel fact that the number of critical points of $\Psi_\eta$ is equal to the number of degeneration points of $\eta$. 

Indeed, it follows from Corollary \ref{fiber} that  for a regular value $a\in \R$ of $\widetilde{ \Phi}_\eta$, the infinite cyclic cover $\widetilde{X}$ corresponding to $\Phi_{\eta}$ is homotopy equivalent to $\widetilde{ \Phi}_\eta^{-1}(a)$ by adding $n$-cells corresponding to the degeneration points of $\eta_X$. Moreover, we have an isomorphism of $\C[t,t^{-1}]$-modules: \be\label{ls} H_n(\widetilde{X},\widetilde{ \Phi}_\eta^{-1}(a);\C)\cong \C[t,t^{-1}]^{\oplus \ell},\ee where $\ell$ is  the number of degeneration points of $\eta_X$. We also have that $H_i(\widetilde{X},\widetilde{ \Phi}_\eta^{-1}(a);\C)\cong 0$, for any $i \neq n$.

Assume now that  $\eta_X \in H^1(X; \R)$ is integral.
Then the assumptions of Theorem \ref{finite} are automatically satisfied, hence  $\widetilde{ \Phi}_\eta^{-1}(a)$ has finite homotopy type.  Let $\xi $ denote the homomorphism $\xi: \pi_1(X) \to 
\Z$ corresponding to $\eta_X$. Then the corresponding infinite cyclic cover $X^\xi$ coincides with $\widetilde{X}$. By (\ref{ls}) and from the arguments in the proof of Proposition \ref{se}, we get that $H_i(X^\xi ; \C)$ is a torsion $\C[t,t^{-1}]$-module for $i\neq n$, and $$\ell= \rk_{\C[t,t^{-1}]} H_n(X^\xi; \C) = (-1)^n \chi (X).$$

The case  when $\eta_X \in H^1(X; \C)$ is non-integral follows from the same argument as in the affine torus case.
%{\color{blue}
%As in Corollary \ref{fiber}, for a regular value $a\in \R$ of $\widetilde{ \Phi}_\eta$,  Theorem \ref{nonpropercircle} implies that $\widetilde{ \Phi}_\eta^{-1}[a, a+1]$ is homotopy equivalent to $\widetilde{ \Phi}_\eta^{-1}(a)$ by adding $n$-cells corresponding to degeneration points of $\eta_X$. Let us denote by $\ell$ the number of degeneration points of $\eta_X$. Then the infinite cyclic cover $\widetilde{X}$ corresponding to $\Phi_{\eta}$ is homotopy equivalent to $\tilde{ \Phi}_\eta^{-1}(a)$ by adding $\Z^r$ $n$-cells. Moreover
%Assume that  $\eta_X \in H^1(X; \R)$ is integral. Then the assumption in Theorem \ref{finite}  is satisfied automatically, hence  $\tilde{ \Phi}_\eta^{-1}(a)$ has finite homotopy type.  Let $\xi $ denote the homomorphism $\xi: \pi_1(X) \to  \Z$ corresponding to $\eta_X$. Then the corresponding infinite cyclic cover $X^\xi$ coincides with $\tilde{X}$. Due to the same proof as in Proposition \ref{se}, $H_i(X^\xi ; \C)$ is a torsion $\C[t,t^{-1}]$-module for $i\neq n$, and $$r= \rk_{\C[t,t^{-1}]} H_n(X^\xi; \C) = (-1)^n \chi (X).$$ The case  $\eta_X \in H^1(X; \C)$ non-integral follows from the same argument as in the affine torus case.} 
\end{proof}

\begin{remark}\rm
The statement (iii) of Theorem \ref{huhgen} follows essentially from \cite{FK00}. One can also replace the index theorem of \cite{FK00} by the Chern class argument in \cite{Hu} to conclude statement (iii). Nevertheless, both proofs require some sophisticated forms of intersection theory extended to the non-compact setting. In contrast, our proof only uses Morse theory. Especially when $G$ is an affine torus, our proof does not rely on a compactification of $G$ or $X$, and is much more elementary. 

%{\color{blue} This is why I am so struggling about the generic finite claim, since the other two proofs both use some intersection theory. Since our proof do not need any intersection theory, it seems like to me that the following claim is also true:

%Let $G$ be a complex algebraic group, hence smooth.  Let $X$ be an $n$-dimensional closed smooth sub-variety of $G$. The Palais-Smale condition holds in this case by the same proof.  Let $\Gamma$ be the space of left invariant holomorphic 1-forms on on $G$. Given $\eta\in \Gamma$, we denote the restriction of $\eta$ to $X$ by $\eta_X$. If $\eta\in \Gamma$ is general, then Theorem \ref{huhgen} (1) and (2) also hold. How about Theorem\ref{huhgen} (3) ?}

\end{remark}

%%%%%%%%%%%%%%%%%%%%%%%%%%%%%%%%%
%%%%%%%%%%%%%%%%%%%%%%%%%%%%%%%%%

\section{Semi-abelian varieites. Proof of structure results}\label{mp}

In this section, we review the definition of complex semi-abelian varieties and discuss some compactifications as complex and real algebraic varieties. We will prove Theorem \ref{finite} using a compactification of $G_\R$, the underlying real algebraic variety underlying $G$. We also supply here a proof of Proposition \ref{representative}.

\medskip

Recall that a commutative complex algebraic group $G$ is called a {\it semi-abelian variety} if there is a short exact sequence of complex algebraic groups
\begin{equation}\label{SES}
1\to T\to G\to A\to 1,
\end{equation}
where $T$ is an affine torus and $A$ is an abelian variety. 
Let $m=\dim T$. As a complex Lie group, an $N$-dimensional semi-abelian variety $G$ is isomorphic to $\C^N/\Lambda$, where $\Lambda\subset \C^N$ is a discrete subgroup generating $\C^N$ as a complex vector space. 

\medskip

Even though the following result is not directly applicable to the proof of Theorem \ref{finite}, it motivates our later constructions. 

\begin{lemma}[\cite{FK00}]\label{compactification}
There exists a %{\color{red}The word "natural" is removed, because the compactification depends on the isomorphism $T\cong (\C^*)^m$} 
projective compactification $\overline{G}$ of $G$, such that $\overline{G}$ is a $\cp^m$-bundle over $A$, which is constructed as a fiberwise $\cp^m$-compactification of $T\cong (\C^*)^m$. %Moreover, $\overline{G}$ is a good compactifcation of $G$, i.e., $\overline{G}$ is smooth and the complement $\overline{G}\setminus G$ is a normal crossing divisor of $\overline{G}$. 
\end{lemma}

\begin{proof}
By choosing a splitting $T\cong \C^*\times \cdots \times \C^*\cong (\C^*)^m$, we can write $G$ as
$$G=G_1\times_A G_2\times_A \cdots \times_A G_m$$
where each $G_k$ is an extension of $A$ by $\C^*$. Then each $G_k$ is a principal $\C^*$-bundle over $A$. Let $L_k$ be the flat line bundle on $A$, whose principal $\C^*$-bundle is isomorphic to $G_k$. In other words, $L_k:=(G_k\times \C)/\C^*$, where $\C^*$ acts on $G_k$ by fiberwise multiplication, and $\C^*$ acts on $\C$ by scalar multiplication. There is a canonical open embedding $G_k\to L_k$, mapping $t\in G_k$ to the image of $(t, 1)$ in the quotient $(G_k\times \C)/\C^*=L_k$. 

By construction, each $L_k$ is an algebraic vector bundle on $A$. Thus, the total space of $L_1+ L_2+ \cdots+ L_m$ has a natural algebraic variety structure. Let $\overline{G}$ be the fiberwise projective compactification of $L_1+ L_2+ \cdots+ L_m$. In other words, $\overline{G}$ is the associated projective bundle of $L_1+ L_2+ \cdots+ L_m+\C_A$, where $\C_A$ is the trivial line bundle on $A$. Thus, $\overline{G}$ is a smooth projective variety. Consider $G$ as an open subvariety of $\overline{G}$ by the composition of natural open embeddings
$$G=G_1\times_A G_2\times_A \cdots \times_A G_m\to L_1+ L_2+ \cdots+ L_m\to \overline{G}.$$
The assertions follow immediately from the above construction. 
\end{proof}

Since the affine torus $T\cong (\C^*)^m$ fits naturally into a short exact sequence of real Lie groups,
$$1\to (\R_{>0})^m\to T\to (S^1)^m\to 1,$$
and since $A\cong (S^1)^{2N-2m}$ as real Lie groups, the short exact sequence (\ref{SES}) induces a short exact sequence 
\begin{equation}\label{realliegroups}
1\to \R_{>0}^m\to G_\R\xrightarrow[]{\sigma} (S^1)^{2N-m}\to 1
\end{equation}
of real Lie groups. %We also recall that $\eta_T$ is the restriction of $\eta$ to $T$. 
In the notations of Section \ref{mr}, we have the following:

\begin{lemma}\label{equivalent}
Given $\eta\in \Gamma$, suppose that the class $[\Re(\eta)]\in H^1(G; \R)$ is integral. Then the following statements are equivalent:
\begin{enumerate}[(i)]
\item the class $[\eta_T]\in H^1(T, \C)$ is integral;
\item the class $[\eta_T]\in H^1(T, \C)$ is real;
\item the left invariant holomorphic $1$-form $\eta_T$ on $T$ is of the form 
$$\eta_T=(2\pi\sqrt{-1})^{-1}\left(n_1\frac{dz_1}{z_1}+\cdots+n_m\frac{dz_m}{z_m}\right)$$
where $n_i\in \Z$, with respect to a choice of coordinates $z_1, \ldots, z_m$ on $T\cong (\C^*)^m$;
\item the map $\Phi_{\eta, G}=\int\Re(\eta): G\to S^1$ contracts the Lie subgroup $(\R_{>0})^m$ of $G$. In other words, $\Phi_{\eta, G}$ factors through $\sigma: G\to (S^1)^{2N-m}$. 
\end{enumerate}
\end{lemma}
\begin{proof}
First, we prove (i)$\iff$ (ii). Since $[\Re(\eta)]\in H^1(G; \R)$ is integral, its restriction to $T$, $[\Re(\eta_T)]\in H^1(T; \R)$ is also integral. Thus, the class $[\eta_T]\in H^1(T; \C)$ is integral if and only if $[\Im(\eta_T)]=0\in H^1(T; \R)$. 

Since $\int\Re(\eta): G\to S^1$ is a real Lie group homomorphism, (iv) holds if and only if $\int\Re(\eta_T): T\to S^1$ contracts the Lie subgroup $(\R_{>0})^m$ of $T$. Both (ii)$\iff$ (iii) and (iii)$\iff$(iv) follow from straightforward computations by using the polar coordinates $(r_i, \theta_i)$. 
\end{proof}

%%%%%%%%%%%%%%%%%%%%%%%%%%%%%

\subsection{Proof of Theorem \ref{finite}}
We first construct a compactification of $X$, such that  the map $\Phi_\eta: X\to S^1$ extends to this compactification. Then we can use the stratified Morse theory of Goresky-McPherson to show that the fibers of $\Phi_\eta$ have finite homotopy type. Notice that $\Phi_\eta: X\to S^1$ is the restriction of $\Phi_{\eta, G}=\int\Re(\eta): G\to S^1$ to $X$. Naturally, we want to construct a compactification of $G$ and let the compactification of $X$ be the closure of $X$ in the compactification of $G$. 

Such compactification of $G$ does not exist in the category of smooth complex algebraic varieties.  In fact, since the fundamental group is a birational invariant for smooth complex projective varieties, by Lemma \ref{compactification} the fundamental group of any smooth compactification of $G$ is isomorphic to $\pi_1(A)$. However, the induced map $(\Phi_\eta)_*: \pi_1(X)\to \pi_1(S^1)$ may not factor through $\pi_1(A)$ in general. Instead, we will construct such a compactification of $G$ as a real algebraic variety. 

In this section, we will frequently consider $G$ and $X$ as real algebraic varieties whose dimensions are twice their complex dimensions. To emphasize the real algebraic variety structures, we will denote them by $G_\R$ and $X_\R$, respectively. We will also use $G_\R$ when we consider $G$ as a real Lie group. 

Since $\R_{>0}^m$ is not a real algebraic group, (\ref{realliegroups}) is not  a short exact sequence of real algebraic groups. Notice that $\C^*$ with the underlying real algebraic group structure fits into a short exact sequence 
$$1\to \R^*\to \C^*\to S^1\to 1$$
of real algebraic groups, where the first arrow is the natural inclusion, and the second map is given by 
$$x+\sqrt{-1}y\mapsto \left(\frac{x^2-y^2}{x^2+y^2}, \frac{2xy}{x^2+y^2}\right).$$
Thus, the composition $(\R^*)^m\to T_\R=(\C^*)^m\to G$ is a homomorphism of real algebraic groups, and we have a short exact sequence of real algebraic groups
\begin{equation}\label{realalgebraicgroups}
1\to (\R^*)^m\to G_\R\xrightarrow[]{\tau} H\to 1, 
\end{equation}
with $H$ a real algebraic group which is isomorphic to $(S^1)^{2N-m}$ as a real Lie group. By taking a fiberwise $(S^1)^m$-compactification of $(\R^*)^m$, we obtain a compactification $\overline{G}'_\R$ of the real algebraic variety $G_\R$. 

Let us now give a more precise construction of $\overline{G}'_\R$. Recall that $G$ is the extension of an abelian variety $A$ by a complex affine torus $T\cong (\C^*)^m$. For the moment, we stay in the category of complex algebraic varieties. As in Lemma \ref{compactification}, we can consider $G$ as a fiber product 
$$G\cong G_1\times_A G_2\times_A \cdots \times_A G_m,$$
where each $G_i$ is a $\C^*$-bundle over $A$. The associated line bundle of $G_i$ was denoted by $L_i$. Let $\overline{G}_i$ be the fiberwise projective compactification of $L_i$, i.e., the associated projective bundle of $L_i+\C_A$. Clearly, the $\cp^1$-bundle $\overline{G}_i
\to A$ has two distinguished sections denoted by $0_A$ and $\infty_A$ corresponding to the fiberwise $0\in \cp^1$ and $\infty\in \cp^1$, respectively. 

Now, we move to the category of real algebraic varieties. Denote the associated real algebraic variety of $\overline{G}_i$ by $\overline{G}_{i\R}$. The sections $0_A, \infty_A\subset \overline{G}_{i\R}$ are real subvarieties of codimension 2. Define $\overline{G}_{i\R}'$ to be the (real) blowup of $\overline{G}_{i\R}$ along both $0_A$ and $\infty_A$. Then, as a real manifold, $\overline{G}_{i\R}'$ is a fiber bundle over $A_\R$, whose fibers are Klein bottles. The short exact sequence of complex algebraic groups
$$1\to \C^*\to G_i\to A\to 1$$
induces a short exact sequence of real algebraic groups
$$1\to \R^*\to G_{i\R}\to H_i\to 1,$$
where $H_i$ is a real algebraic group which is isomorphic to $(S^1)^{2N-2m+1}$ as a real Lie group. Clearly, $\overline{G}'_{i\R}$ is a compactification of $G_{i\R}$ as a real algebraic variety, which fiberwise is a $S^1$-compactification of $\R^*$. Now, define
$$\overline{G}'_\R:=\overline{G}'_{1\R}\times_{A_\R}\cdots\times_{A_\R}\overline{G}'_{m\R}.$$
Then $\overline{G}'_\R$ is a compactification of $G_\R$ as a real algebraic variety, which fiberwise is a $(S^1)^m$-compactification of $(\R^*)^m$. 

%We can give another explicit description of the algebraic structure of the fiber of $\overline{G}_{i\R}'\to A_\R$. We started with $\C^*=\{x+iy|x, y\in \R, x^2+y^2\neq 0\}$. Consider $\C^*$ as a $\R^*$-bundle over $S^1$ by the real algebraic map 
%\begin{equation}\label{algebraicprojection}
%\C^*\to S^1, (x, y)\mapsto \left(\frac{x^2-y^2}{x^2+y^2}, \frac{2xy}{x^2+y^2}\right).
%\end{equation}
%Now we take a fiberwise $S^1$ compactification of $\R^*$. Then each fiber becomes a Klein bottle. Taking this compactification on each fiber of $G_\R\to A$, we obtain the compact real algebraic variety $\overline{G}_{i\R}'$. 

%Now, let 
%$$\overline{G}_\R=\overline{G}_{1\R}'\times_{A_\R}\overline{G}_{2\R}'\times_{A_\R}\cdots\times_{A_R}\overline{G}_{m\R}'.$$
%Then $\overline{G}_\R$ is a smooth compactification of the real algebraic variety $G_\R$. Recall that $\eta$ is a logarithmic 1-form on $\overline{G}$ whose restriction to $G$ is a left invariant holomorphic 1-form. Moreover, we assumed that $[\Re\eta]\in H^1(G, \R)$ is integral. Thus, by a consistent choice of the base point, we have a smooth map $\Phi_{\eta, G}:=\int \Re(\eta): G\to S^1$, whose restriction to $X$ is equal to $\Phi_\eta: X\to S^1$. 

In general, the map $\Phi_{\eta,G}$ may not extend to a real analytic map $\overline{G}'_\R\to S^1$, because the homomorphism $\tau: G_\R\to H$ in (\ref{realalgebraicgroups}) does not induce an isomorphism on the fundamental groups. In fact, the induced map $\tau_*: \pi_1(G_\R)\to \pi_1(H)$ is injective, and its cokernel is a 2-torsion group. %However, the map $2\Phi_\eta: G\to S^1$ does. 

\begin{lemma}
Under the assumptions of Theorem \ref{finite}, the map $2\Phi_{\eta, G}:=\int2\Re(\eta): G\to S^1$ extends to a real analytic map $\overline{2\Phi_{\eta, G}}: \overline{G}'_\R\to S^1$. 
\end{lemma}
\begin{proof}
By Lemma \ref{equivalent}, $\Phi_{\eta, G}=\int\Re(\eta): G\to S^1$ contracts the fibers $(\R_{>0})^m$ and factors through the map $\sigma: G=G_\R\to (S^1)^{2N-m}$ in (\ref{realliegroups}). By definition, the map $\tau: G_\R\to H$ in (\ref{realalgebraicgroups}) factors through $\sigma: G=G_\R\to (S^1)^{2N-m}$, and $2\sigma$ factors through $\tau$. Thus, the map $2\Phi_{\eta, G}$ factors through $\tau$. Note that $\tau$ is a real algebraic map (although $\sigma$ is not).  Since with respect to (\ref{realalgebraicgroups}) the real algebraic variety $\overline{G}'_\R$ is a fiberwise $(S^1)^m$-compactification of $(\R^*)^m$, the map $2\Phi_{\eta, G}:=\int2\Re(\eta): G\to S^1$ extends to a real analytic map $\overline{2\Phi_{\eta, G}}: \overline{G}'_\R\to S^1$. 
%By choosing  $T$, the short exact sequence
%$$0\to T\to G\to A\to 0$$
%splits non-canonically into 
%$$G_\R\cong T_\R\times A_\R$$
%as real Lie groups, where $m=\dim T$ and $N=\dim G$. Thus, non-canonically, we have a short exact sequence of real Lie groups
%\begin{equation}\label{realliegroups1}
%0\to (\R_{>0})^m\to G_\R\xrightarrow[]{\tau} (S^1)^{2N-m}\to 0.
%\end{equation}
%Using the algebraic map (\ref{algebraicprojection}), $T_\R$ is naturally a $(\R^*)^m$ bundle over $(S^1)^m$. Thus, we have non-canonically a short exact sequence of real Lie groups
%\begin{equation}\label{realliegroups2}
%0\to (\R^*)^m\to G_\R\xrightarrow[]{\sigma} (S^1)^{2N-m}\to 0.
%\end{equation}
%Clearly, the map $\sigma$ factors through $\tau$, and $2\tau$ factors through $\sigma$. Considering the short exact sequence (\ref{realliegroups2}) as a fiber bundle, our construction of $\overline{G}_\R$ is a fiberwise $(S^1)^m$-compactification. 
%Since $\Phi_{\eta, G}: X\to S^1$ is defined by integrating a left invariant 1-form, it is a real Lie group homomorphism, and hence it factors through the Lie group homomorphism $\tau: G_\R\to (S^1)^{2N-m}$ in the short exact sequence (\ref{realliegroups1}). Therefore, the map $2\Phi_{\eta, G}$ factors through $\sigma$, and hence it extends to a real analytic map $\overline{2\Phi_{\eta,G}}: \overline{G}_\R\to S^1$. 
\end{proof}

Now, since $\overline{G}'_\R$ is a compactification of $G_\R$ as a real algebraic variety, the closure $\overline{X}_\R$ of $X_\R$ in $\overline{G}'_\R$ is a compactification of $X_\R$ as a real algebraic variety. Hence the complement $\overline{X}_\R\setminus X_\R$ is a real Zariski closed set in $\overline{X}_\R$. Since $\overline{2\Phi_{\eta,G}}: \overline{G}'_\R\to S^1$ is a real analytic map, for any $b\in S^1$, the corresponding fiber $\overline{2\Phi_{\eta,G}}^{-1}(b)$ is a closed real analytic subvariety of $\overline{G}'_\R$. Thus, for any $a\in S^1$, $\Phi_\eta^{-1}(a)$ consists of some connected components of 
$$X\cap \overline{2\Phi_{\eta,G}}^{-1}(2a)=\left(\overline{X}_\R\cap \overline{2\Phi_{\eta,G}}^{-1}(2a)\right)\setminus \left(\overline{G}'_\R\setminus G_\R\right),$$
which is a real analytic constructible set of a compact real analytic manifold. By the stratified Morse theory of Goresky-MacPherson (see \cite[Theorem 2.2.1 and Section 10.8]{GM88}), the real analytic constructible set $X\cap \overline{2\Phi_{\eta,G}}^{-1}(2a)$ is homotopy equivalent to a finite CW-complex. Therefore, as it consists of some of the connected components of the real analytic constructible $X\cap \overline{2\Phi_{\eta,G}}^{-1}(2a)$, the subset $\Phi_\eta^{-1}(a)$ is also homotopy equivalent to a finite CW-complex. This completes the proof of Theorem \ref{finite}. $\hfill\square$

%%%%%%%%%%%%%%%%%%%%%%%%%%%%%%%%%

\subsection{Proof of Proposition \ref{representative}}

We first make the following simple observation. 
\begin{lemma}
The short exact sequence (\ref{SES}) as a complex of complex Lie groups induces a canonical splitting of $G_\R$ as a direct sum of two real Lie groups
\begin{equation}\label{splitting}
G_\R\cong \R^m\oplus (S^1)^{2N-m}.
\end{equation}
\end{lemma}
\begin{proof}
As a real Lie group, $G_\R\cong \R^{2N}/\Lambda$, where $\Lambda$ is a discrete subgroup of $\R^{2N}$ of rank $2N-m$. Let $\Lambda_\R$ be the subspace of $\R^{2N}$ generated by $\Lambda$. Then $\Lambda_\R/\Lambda\cong (S^1)^{2N-m}$ is a Lie subgroup of $G_\R$. Obviously, the restriction of $\sigma: G_\R\to (S^1)^{2N-m}$ to $\Lambda_\R/\Lambda$ is an isomorphism. Thus, the inclusion map $\Lambda_\R/\Lambda\to G$ gives rise to a splitting of the short exact sequence (\ref{realliegroups}). Notice that $\R$ is isomorphic to $\R_{>0}$ as real Lie groups via the exponential map. Thus, the splitting (\ref{splitting}) is canonical.
\end{proof}

\begin{proof}[Proof of Proposition \ref{representative}]
Let $H$ be the space of left invariant real 1-forms on $G_\R$. The splitting (\ref{splitting}) induces complementary subspaces $H_{\R^m}$ and $H_{(S^1)^{2N-m}}$ of $H$. 

Since every left invariant 1-form on $G_\R$ is closed, we have a linear map 
\begin{equation}\label{class}
\zeta\mapsto [\zeta]: H\to H^1(G; \R).
\end{equation}
By (\ref{splitting}), the kernel of (\ref{class}) is equal to $H_{\R^m}$, and the restriction of (\ref{class}) to $H_{(S^1)^{2N-m}}$ is an isomorphism. 

Notice that the complex structure of $G$ gives rise to an action on $H$, which we denote by $J$. Then the map 
$$\xi\mapsto \xi+J\xi: H\to \Gamma$$
is an inverse of the map $\Re: \Gamma\to H$. Furthermore, it follows from a polar coordinate calculation as in the proof of Lemma \ref{equivalent} that for $\eta\in \Gamma$, $[\eta_T]$ is real if and only if $\Re(\eta)\in H_{(S^1)^{2N-m}}$. Therefore, the restriction of the maps 
$$ \Gamma\xrightarrow[]{\Re} H\xrightarrow[]{[-]}H^1(G; \R)$$
induces two isomorphisms of real vector spaces 
$$\{\eta\in \Gamma| \ [\eta_T]\in H^1(T, \C) \text{ is real}\}\to H_{(S^1)^{2N-m}}\to H^1(G; \R),$$ thus proving the claim.
%Let $\sg$ (resp. $\sg_\R$) be the complex (resp. real) Lie algebra of the complex (resp. real) Lie group $G$ (resp. $G_\R$). Let $J: \sg_\R\to \sg_\R$ be the action of the complex structure of $G$. Then $J$ induces an action on the real cotangent space $T_e^*G_\R$. Let $H$ be the space of left invariant real 1-forms on $G_\R$. Then the parallel transport defines a canonical isomorphism between $T_e^*G_\R$ and $H$. Thus, by abusing notation, we can consider $J$ as an action on $H$. Now, taking the real part $\Re: \Gamma\to H$ is an isomorphism of real vector spaces, because $\zeta\mapsto \zeta+J\zeta: H\to \Gamma$ is its inverse map. 
\end{proof}

%%%%%%%%%%%%%%%%%%%%%%%%%%%%%%%%%%%%%%%%%%%%%%%%%%%%%%%%%%

\section{Palais-Smale conditions}\label{PS}
In this section we prove Theorem \ref{main} and Theorem \ref{veryaffine}.
%{\color{magenta}The complex compactification of $G$ is not necessary. The proof is simplified. }

First, we construct the Zariski open subset $U$ in Theorem \ref{main}. 
Recall that $\Gamma$ is the space of left invariant holomorphic $1$-forms on $G$, $\dim G=N$ and $\dim X=n$. Define the degenerating locus\footnote{The definition of the degenerating locus $Z$ is motivated by the likelihood correspondence in algebraic statistics (see \cite{HS}). } in $X\times \Gamma$ by
$$Z:=\{(x, \eta)\in X\times \Gamma| \ \eta \text{ vanishes at } x\}.$$
Since the degenerating condition is given by algebraic equations, the locus $Z$ is an algebraic set of $X\times \Gamma$. Since $X$ is smooth, the first projection $p_1: Z\to X$ is a vector bundle with fiber dimension $N-n$. Thus, $\dim Z=N$, and the second projection $p_2: Z\to \Gamma$ is generically finite (e.g., by using again Verdier's generic fibration theorem). Let $U\subset \Gamma$ be the Zariski open set where $p_2: Z\to \Gamma$ is finite and \'etale, or, equivalently, a finite covering map. Since being finite and \'etale is an open condition, $U\subset \Gamma$ is a non-empty Zariski open subset. %Moreover, $U$ is clearly non-empty. 

One can easily check that the holomorphic $1$-form $\eta_X$ degenerates at a point $x$ if and only if $Z$ intersects $X\times \{\eta\}$ at $(x, \eta)$. Moreover, the $1$-form $\eta_X$ has a regular singularity at $x$ if and only if $Z$ intersects $X\times \{\eta\}$ transversally at $(x, \eta)$. Since the projection $p_2: Z\to \Gamma$ is \'etale over $U$, and since $\eta \in U$, it follows that $Z$ intersects $X\times \{\eta\}$ transversally. Thus, the $1$-form $\eta_X$ has a regular singularity at every degeneration point. By the holomorphic Morse lemma, the real part $\Re(\eta_X)$ also has a regular singularity at every degeneration point. Thus, the function $\Phi_\eta=\int\Re(\eta_X)$ is a circle-valued Morse function. 

The following lemma follows immediately from the fact that $p_2$ is finite covering map over $U$.
\begin{lemma}\label{etale}
For any $\eta\in U$, there exists a small open neighborhood $U_\eta$ (in the sense of analytic topology) of $\eta$ in $U$, such that $\overline{U}_\eta\subset U$ and $p_2^{-1}(\overline{U}_\eta)\subset X\times \overline{U}_\eta$.  Here $\overline{U}_\eta$ denotes the closure of $U_\eta$ in $\Gamma$. In particular, $p_2^{-1}(\overline{U}_\eta)\cap (X\times\overline{U}_\eta)=p_2^{-1}(\overline{U}_\eta)$ is compact. 
\end{lemma}

Let $\sg_\R$ be the real Lie algebra of the real Lie group $G_\R$. Denote the real dual vector space of $\sg_\R$ by $\sg_\R^\vee$. Let $h_e$ be a Hermitian metric on the vector space $\mathfrak{g}^\vee_\C=\sg^\vee_\R\otimes_\R\C$, and let $h$ be the left invariant Hermitian metric on $G$ determined by $h_e$. Denote the restriction of $h$ to $X$ by $h_X$, and denote the associated Riemannian metric of $h_X$ by $g_X$. Since $\Phi_\eta: X\to S^1$ is defined as $\int \Re(\eta_X)$, we have $d\Phi_\eta=\Re(\eta_X)$. We claim the following:
\begin{equation}\label{equalities}
\begin{split}
||\eta_X||^2_{h_X}&=||\Re (\eta_X)||^2_{g_X}+||\Im (\eta_X)||^2_{g_X}\\
&=2||\Re (\eta_X)||^2_{g_X}\\
&=2||\nabla \Phi_\eta||^2_{g_X}.
\end{split}
\end{equation}
In fact, the first equality holds since $\eta_X$ is a holomorphic $1$-form. The third equality follows from the definition of $\Phi_\eta$. To check the second equality at a point $x\in X$, let $z_i=x_i+\sqrt{-1}y_i$ be a set of unitary coordinates of $X$ at $x$. Since $\eta_X$ is a holomorphic $1$-form, by the Cauchy-Riemann equation we have:
$$\langle \Re(\eta_X), \partial/\partial_{x_i}\rangle=\langle\Im(\eta_X), \partial/\partial_{y_i}\rangle$$
and
$$\langle\Re(\eta_X), \partial/\partial_{y_i}\rangle=-\langle\Im(\eta_X), \partial/\partial_{x_i}\rangle.$$
Thus, at the point $x\in X$,
\begin{equation*}
\begin{split}
||\Im (\eta_X)||^2_{g_X}&=\sum_{i}\left(\langle\Im(\eta_X), \partial/\partial_{x_i}\rangle^2+\langle\Im(\eta_X), \partial/\partial_{y_i}\rangle^2\right)\\
&=\sum_{i}\left(\langle\Re(\eta_X), \partial/\partial_{y_i}\rangle^2+\langle\Re(\eta_X), \partial/\partial_{x_i}\rangle^2\right)\\
&=||\Re (\eta_X)||^2_{g_X}.
\end{split}
\end{equation*}
Since $x$ is an arbitrary point in $X$, equality (\ref{equalities}) holds.

 \begin{proof}[Proof of Theorem \ref{main}]
Suppose $S$ is a subset of $X$ on which $||\nabla \Phi_\eta||^2_{g_X}$ is not bounded away from zero, and assume that the closure of $S$ does not contain any degeneration point of $\eta_X$. Since $||\nabla \Phi_\eta||^2_{g_X}=\frac{1}{2}||\eta_X||^2_{h_X}$, it follows that there exists a sequence of points $(x_i)_{i=1}^{\infty}$ in $S$ such that $\lim_{i\to \infty}||(\eta_X)_{x_i}||_{h_X}=0$ and $(x_i)_{i=1}^\infty$ does not have any limit point in $X$. 

To finish the proof, we need the following lemma.
\begin{lemma}\label{42}
Suppose $||(\eta_X)_{x}||_{h_X}=\epsilon$ for some point $x\in X$. Then there exists a unique $\theta \in \Gamma$ such that $||\theta||_{h_\Gamma}=\epsilon$ and $\eta_X-\theta_X$ degenerates at $x$. Here, the norm $||\cdot||_{h_\Gamma}$ on $\Gamma$ is defined to be the restriction of the norm $||\cdot||_{h_e}$ on $\sg_\C^\vee$ via the natural isomorphisms $\Gamma\cong\sg^\vee\cong \mathfrak{g}_\C^{\vee(1,0)}$, where $\sg$ is the Lie algebra of the complex Lie group $G$, and $\sg^\vee$ is the dual of $\sg$ as a complex vector space. 
\end{lemma}

\begin{proof}
The statement is local at $x$ and only involves the first order differential. Without loss of generality, we can assume that $G=\C^N$ with coordinates $(z_j)_{1\leq j\leq N}$ and the standard Hermitian metric $h$, $X=\{z_{n+1}=\cdots=z_N=0\}$, and $x=(0, \ldots, 0)$. Since $||(\eta_X)_{x}||_{h_X}=\epsilon$, by a unitary coordinate change within $(z_j)_{1\leq j\leq n}$, we can assume that 
$$\eta_X=\eta_1dz_1+\cdots+\eta_ndz_n,$$
where $\eta_j$ are local holomorphic functions on $X$ near $x$ with $\eta_1(x)=\epsilon$, and $\eta_j(x)=0$ for $2\leq j\leq n$. Since $\theta$ is an left invariant holomorphic $1$-form on $G$, it is of the form 
$$\theta=\theta_1 dz_1+\cdots+\theta_N dz_N,$$
% There seem to be issues with this claim
where $\theta_j\in \C$ are constant. The condition that $\eta_X-\theta_X$ degenerates at $x$ is equivalent to 
$\theta_1=\epsilon$ and $\theta_j=0$ for $2\leq j\leq n$. Now, the condition $||\theta||_{h_\Gamma}=\epsilon$ implies that $\theta_j=0$ for all $j>n$. Thus, $\theta=\epsilon dz_1$ is the unique left invariant holomorphic $1$-form satisfying the conditions in the lemma. 
\end{proof}

By Lemma \ref{42}, for each $i$ we can find $\theta_i\in \Gamma$ such that $||\theta_i||_{h_\Gamma}=||(\eta_X)_{x_i}||_{h_X}$ and $\eta_X-(\theta_i)_X$ degenerates at $x_i$. Thus, we have a sequence of points $(x_i, \eta-\theta_i)$ in $Z$. Since $\lim_{i\to \infty}||(\eta_X)_{x_i}||_{h_X}=0$, we have $\lim_{i\to \infty}||\theta_i||_{h_\Gamma}=0$, or equivalently, $\eta-\theta_i$ converges to $\eta$ in $\Gamma$. Since $x_i$ does not have a limit point in $X$, the sequence $(x_i, \eta-\theta_i)$ does not have a limit point in $X\times \Gamma$. However, by Lemma \ref{etale}, there exists an open neighborhood $U_\eta$ of $\eta$ in $\Gamma$ such that $p_2^{-1}(\overline{U}_\eta)\cap (X\times \overline{U}_\eta)$ is compact. Since $\eta-\theta_i$ converges to $\eta$, and since $\eta\in U$, when $i$ is sufficiently large we have $\eta-\theta_i\in U$. Hence, by Lemma \ref{etale}, when $i$ is sufficiently large, $(x_i, \eta-\theta_i)$ is contained in the compact subset $p_2^{-1}(\overline{U}_\eta)\cap (X\times \overline{U}_\eta)$ of $Z$. This is a contradiction to the fact that $x_i$ does not have a limit point in $X$. 

Thus, we have concluded the proof of Theorem \ref{main}. 
\end{proof}

The proof of Theorem \ref{veryaffine} is very similar to that of Theorem \ref{main} and it will be sketched below. %So we will only give sketch of the proof
\begin{proof}[Proof of Theorem \ref{veryaffine}]
First, we should replace $\C^*$ by $\C/\Z$, and write $\Psi_\eta=\int\eta_X: X\to \C/\Z$ in the theorem to be more precise. The map $\int\eta_X: X\to \C/\Z$ splits into the real part $\int\Re(\eta_X): X\to S^1=\R/\Z$ and the imaginary part $\int\Im(\eta_X): X\to \R$. Composing with the isomorphism $\C/\Z\cong \C^*$ defined by $z\mapsto e^{2\pi\sqrt{-1}z}$, the map $|\cdot|: \C^*\to \R_{>0}$ becomes $z\mapsto e^{2\pi\Im(z)}: \C/\Z\to \R_{>0}$. Hence, 
$$\Psi_\eta(x)=\log \Big|\int_{e}^x\eta_X\Big|=\log\left(e^{2\pi\Im\int^x_{e}\eta_X}\right)=2\pi\int^x_{e}\Im(\eta_X)$$
where $\int_{e}^x\eta_X\in\C^*$ and $\int^x_{e}\Im(\eta_X)\in\R$. In other words, $\Psi_\eta=2\pi\int\Im(\eta_X)$. Since $\eta_X$ is a holomorphic $1$-form, by equation (\ref{equalities}) we get:
$$||\nabla\Psi_\eta||=2\pi||\Im(\eta_X)||=\sqrt{2}\pi||\eta_X||.$$
Suppose $S$ is a subset of $X$ on which $||\nabla\Psi_\eta||^2_{g_X}$ is not bounded away from zero, and suppose that the closure of $S$ does not contain any degeneration point of $\eta_X$. Then we can find a sequence $(x_i)_{i=1}^\infty$ in $S$ such that $\lim_{i\to \infty}||(\eta_X)_{x_i}||_{h_X}=0$ and $(x_i)_{i=1}^\infty$ does not have any limit point in $X$. Now, we are in the exact same situation as in the proof of Theorem \ref{main}, hence we can use the same argument to conclude the proof. 
\end{proof}

\begin{remark}\rm
In the proof of Theorem \ref{veryaffine}, we did not require $||\Psi_\eta||$ to be bounded on $S$. In other words, the function $\Psi_\eta: X\to \R$ satisfies the following stronger condition: {\it if $S$ is a subset of $X$ on which $||\nabla\Psi_\eta||$ is not bounded away from zero, then there exists a critical point of $\Psi_\eta$ in the closure of $S$.} 
\end{remark}

%%%%%%%%%%%%%%%%%%%%%%%%%%%%%%

\end{document}